\documentclass[10pt,a4paper]{article}
\usepackage{amssymb}
\usepackage{graphicx}
\usepackage{amsmath}
\usepackage{makeidx}
\usepackage{indentfirst}
\usepackage[T1]{fontenc}
\usepackage[utf8]{inputenc}
\usepackage{authblk}

\newcounter{resultnum}[section]
\newtheorem{conclusion}{Conclusion}[section]

\newcounter{conclusionnum}[section]
\newcounter{conditionnum}[section]
\newcounter{conjecturenum}[section]
\newtheorem{example}{Example}[section]

\newcounter{examplenum}[section]
\newcounter{exercisenum}[section]
\newtheorem{lemma}{Lemma}[section]

\newcounter{lemmanum}[section]
\newcounter{notationnum}[section]
\newtheorem{theorem}{Theorem}[section]

\newcounter{theoremnum}[section]
\newtheorem{definition}{Definition}[section]

\newcounter{definitionnum}[section]
\newtheorem{corollary}{Corollary}[section]

\newcounter{corollarynum}[section]
\newtheorem{remark}{Remark}[section]

\newcounter{remarknum}[section]
\newtheorem{proposition}{Proposition}[section]

\newcounter{propositionnum}[section]
\newcounter{acknowledgementnum}[section]
\newcounter{algorithmnum}[section]
\newcounter{axiomnum}[section]
\newcounter{casenum}[section]
\newcounter{claimnum}[section]
\newcounter{summarynum}[section]
\newcounter{problemnum}[section]
\newenvironment{proof}[1][]{\textbf{Proof.} }{}

\begin{document}

\title{Almost K\"{a}hler Ricci Flows and Einstein and Lagrange--Finsler
Structures on Lie Algebroids}
\date{September 1, 2014}

\author{{\large \textbf{Sergiu I. Vacaru}}\thanks{sergiu.vacaru@uaic.ro}\\
{\qquad } \\
{\small Theory Division, CERN, CH-1211, Geneva 23, Switzerland} \footnote{associated visiting research}\\
{\small and } \\
  {\small  Rector's Office, Alexandru Ioan Cuza University}, \\
{\small Alexandru Lapu\c sneanu street, nr. 14, UAIC -- Corpus R, office 323};\\
 {\small Ia\c si,\ Romania, 700057}
  }

\renewcommand\Authands{and }

\maketitle

\begin{abstract}
In this work we investigate Ricci flows of almost K\"{a}hler structures on
Lie algebroids when the fundamental geometric objects are completely
determined by (semi) Riemannian metrics, or (effective) regular generating
Lagrange/ Finsler functions. There are constructed canonical almost
symplectic connections for which the geometric flows can be represented as
gradient ones and characterized by nonholonomic deformations of Grigory
Perelman's functionals. The first goal of this paper is to define such
thermodynamical type values and derive almost K\"{a}hler -- Ricci geometric
evolution equations. The second goal is to study how fixed Lie algebroid,
i.e. Ricci soliton, configurations can be constructed for Riemannian
manifolds and/or (co) tangent bundles endowed with nonholonomic
distributions modelling (generalized) Einstein or Finsler -- Cartan spaces.
Finally, there are provided some examples of generic off--diagonal solutions
for Lie algebroid type Ricci solitons and (effective) Einstein and
Lagrange--Finsler algebroids.

\vskip0.1cm


\end{abstract}

%\tableofcontents

\section{Introduction}

Various theories of geometric flows have been studied intensively in the
past decade. The most popular is the Ricci flow theory \cite{ham1,ham2,ham3}
finally elaborated by G. Perelman \cite{gper1,gper2,gper3} in the form which
allowed proofs of the Thurston and Poincar\'{e} conjectures (details and
proofs are given in \cite{caozhu,kleiner,rbook}). The main results are
related to the evolution of Riemannian and K\"{a}hler metrics and symplectic
curvature flows, see additional references in \cite{chau,streets}.

It should be also emphasized that the Ricci flow theory has an increasing
impact in physical mathematics. In a more general context, we were
interested to study Ricci flow evolution models for non--Riemannian
geometries, for instance, of nonholonomic manifolds endowed with compatible
metric and nonlinear connection structures \cite{vricci1}, metric compatible
and noncompatible Riemann--Finsler--Lagrange spaces \cite{vricci2,vfinsler},
noncommutative geometries and generalizations \cite{vncricci}, fractional
and/or diffusion Ricci flows \cite{vfracrf} etc.

The purpose of this article is to formulate a geometric approach for the
almost K\"{a}hler Ricci flows on/of Lie algebroids when the fundamental
geometric objects are completely defined by a regular Lagrange, $L,$
generating function, or a (generalized) Einstein metric $\mathbf{g}$ with
conventional integer $(n+m)$--dimensional splitting. Similar constructions
can be performed for certain effective (analogous) models and/or, in
particular, for Finsler, $F$, fundamental functions.\footnote{%
In this paper, a nonholonomic manifold $\left( V,\mathcal{N}\right) ,$ [$%
\dim V=n+m$ with finite $n,m\geq 2]$ is modelled as a (semi) Riemannian one
of necessary smooth class and endowed with a non--integrable (nonholonomic,
equivalently, anholonomic) distribution $\mathcal{N}.$ For geometric models
of Lagrange--Finsler spaces, we can work with nonholonomic vector/ tangent
bundles, when $V=TM$ is the total space of the tangent bundle on a real
manifold $M,\dim M=n,$ see details and references in \cite{vrflg}.}

Here is an outline of the work. In section \ref{s2}, we survey the geometry
of Lie algebroids and prolongations endowed with nontrivial N--connection
structure and formulate an approach to the almost K\"{a}hler geometry on
nonolonomic Lie albebroids. The Ricci flow theory for almost symplectic
geometries determined by canonical metric compatible algebroid connections
is studied in section \ref{s3}. Finally, we provide a series of examples of
almost K\"{a}hler Ricci soliton solutions, related to (modified) gravity
models and Lagrange--Finsler Lie algebroid mechanics in section \ref{s4}.
The Appendix contains some necessary and important coefficient type formulas
and proofs.

\section{Almost K\"{a}hler Lie Algebroids \& N--connections}

\label{s2}In this section, we recall basic definitions on noholonomic Lie
algebroids endowed with N--connection structure, see \cite%
{vcliffalg,vhamlagralg}, and develop the approach for almost K\"{a}%
hler Lie algebroids.

\subsection{Preliminaries}

To begin, let us fix a real manifold $V$ of dimension $n\geq 2$ and
necessary smooth class and consider a conventional horizontal (h) and
vertical (v) splitting determined by a Whitney sum for its tangent bundle $%
TV,$
\begin{equation}
\mathbf{N}:\ TV=hTV\oplus vTV.  \label{whitn}
\end{equation}

\begin{definition}
A $h-v$ splitting $\mathbf{N}$ (\ref{whitn}) defines a nonlinear connection,
N--connection, structure.
\end{definition}

We shall use boldface symbols in order to emphasize that the geometric
objects on a space $\mathbf{V}=(V,\mathbf{N,g)}, \mathbf{TV},$ or $\mathbf{TM%
}=(TM\mathbf{,N,}L)$, are adapted to a N--splitting (\ref{whitn}).\footnote{%
\textbf{Notational Remarks: } All constructions in this work can be
performed in coordinate free form. Nevertheless, to find/generate explicit
examples of solutions of systems of partial differential equations, PDEs, is
necessary to consider some adapted frame and coordinate systems. Working
with different classes of geometries and spaces, it is more convenient to
treat indices also as abstract labels. This simplifies formulas and sugests
ideas how certain proofs are possible when local constructions have an
"obvious" global extension. For instance, index type formulas are largely
used in G. Perelman's works and for applications in mathematical relativity
and/or geometric mechanics.
\par
For a conventional $n+m$ splitting, the local coordinates $u=(x,y)$ can be
labelled in the form $u^{\alpha }=(x^{i},y^{a}), $ where $i,j,k,...=1,2,...n$
and $a,b,c...=n+1,n+2,...,n+m,$ where $x^{i}$ and $y^{a}$ are respectively
the h-- and v--coordinates. A general local base is written $e_{\alpha
^{\prime }}=(e_{i^{\prime }},e_{a^{\prime }})$ for some frame transforms $%
e_{\alpha ^{\prime }}=e_{\ \alpha ^{\prime }}^{\alpha }(u)\partial _{\alpha
},$ where $\partial _{\alpha }=\partial /\partial u^{\alpha }=(\partial
_{i}=\partial /\partial x^{i},\partial _{a}=\partial /\partial y^{a}),$ and
define corresponding dual transforms with inverse matrices $e_{\alpha \ }^{\
\alpha ^{\prime }}(u),$ when $e^{\alpha ^{\prime }}=e_{\alpha \ }^{\ \alpha
^{\prime }}(u)du^{\alpha }.$ Calligraphic symbols will be considered for
algebroid configurations. The left "up--low" indices will be considered as
abstract labels without summation rules. Finally, we note that the Einstein
summation rule on "low - up" repeating indices will be applied if the
contrary is not stated.} In global form, the concept of N--connection was
formulated by C. Ehresmann \cite{ehresmann} in 1955. E. Cartan used such a
geometric object in his works on Finsler geometry beginning 1935 (see \cite%
{vrflg} for references and an alternative definition of N--connections via
exact sequences) but in coordinate form, $\mathbf{N}=N_{i}^{a}(x,y)dx^{i}%
\otimes \partial /\partial y^{a}.$ On (semi) Riemannian manifolds, a $h-v$%
--splitting of type (\ref{whitn}) can be defined by a corresponding class of
nonholonomic frames.

Let us consider a nonholonomic distribution on $V$ defined by a generating
function $\mathcal{L}(u),$ with nondegenerate Hessian $\tilde{h}_{ab}=\frac{1%
}{2}\frac{\partial ^{2}\mathcal{L}}{\partial y^{a}\partial y^{b}},\ \det |%
\tilde{h}_{ab}|\neq 0.$ If $V=TM,$ and $n=m,$ we can consider $\mathcal{L}%
=L(x,y)$ as a regular Lagrangian defining a Lagrange space \cite{kern}. For $%
L=F^{2}(x,y),$ where $F(x,\xi y)=\xi F(x,y),\xi >0$, we get a homogeneous
Finsler generating function (additional conditions are imposed on $F$ for
different models of Finsler like geometries), see section \ref{ssfs}. We
shall work with arbitrary $\mathcal{L}$ considering the Finsler
configurations as certain particular ones distinguished by homogeneity
conditions and additional assumptions.

\begin{theorem}
\label{theoranalog}The Euler--Lagrange equations $\frac{d}{d\tau }\frac{%
\partial \mathcal{L}}{\partial y^{i}}-\frac{\partial \mathcal{L}}{\partial
x^{i}}=0,$ where $y^{i}=dx^{i}/d\tau $, for $x^{i}(\tau )$ depending on real
parameter $\tau $, are equivalent to the "nonlinear" geodesic equations $%
\frac{dx^{a}}{d\tau }+2\tilde{G}^{a}(x,y)=0,$ i.e. to the paths of a
canonical semispray $S=y^{i}\frac{\partial \mathcal{L}}{\partial x^{i}}-2%
\tilde{G}^{a}\frac{\partial }{\partial y^{a}},$ when $\tilde{G}^{a}=\frac{1}{%
4}\ \tilde{h}^{a\ n+i}(\frac{\partial ^{2}\mathcal{L}}{\partial
y^{n+i}\partial x^{k}}y^{n+k}-\frac{\partial \mathcal{L}}{\partial x^{i}}),$
where $\tilde{h}^{ab}$ is inverse to $\ \tilde{h}_{ab},$ and the set of
coefficients $\tilde{N}_{i}^{a}=\frac{\partial \tilde{G}^{a}}{\partial
y^{n+i}}$ defines a N--connection structure (\ref{whitn}).
\end{theorem}

An explicit proof consists from straightforward computations. We put "tilde"
on some geometric objects in order to emphasize that they are generated by $%
\mathcal{L}.$

\begin{corollary}
Any prescribed generating function $\mathcal{L}$ on $\mathbf{V}$ defines a
canonical distinguished metric, d--metric $\mathbf{\tilde{g}}$, which is
adapted to the N--connection splitting determined by data $\tilde{h}_{ab}$
and $\tilde{N}_{i}^{a}$, i.e. $\mathbf{\tilde{g}}=h{\tilde{g}}+v{\tilde{g}}$%
, when
\begin{eqnarray}
\mathbf{\tilde{g}} &=&\tilde{g}_{ij}dx^{i}\otimes dx^{j}+\tilde{h}_{ab}%
\mathbf{\tilde{e}}^{a}\otimes \mathbf{\tilde{e}}^{b},\ \tilde{g}_{ij}=\tilde{%
h}_{n+i\ n+j},  \label{lfsm} \\
\mathbf{\tilde{e}}_{\alpha } &=& (\mathbf{\tilde{e}}_{i}=\partial _{i}-%
\tilde{N}_{i}^{a}\partial _{a},e_{a}=\partial _{a}),\ \mathbf{\tilde{e}}%
^{\alpha }=(e^{i}=dx^{i},\mathbf{\tilde{e}}^{a}=dy^{a}+\tilde{N}%
_{i}^{a}dx^{i})  \label{dbases}
\end{eqnarray}
\end{corollary}

The proof is a similar to that for the Sasaki lift from $M$ to $TM,$ see
\cite{yano,matsumoto}.

\subsection{Distinguished Lie algebroids and prolongations}

In Refs. \cite{vcliffalg,vhamlagralg}, we introduced

\begin{definition}
A Lie distinguished algebroid (d--algebroid) $\mathcal{E}=(\mathbf{E}%
,\left\lfloor \cdot ,\cdot \right\rfloor ,\rho )$ over a manifold $M$ is
defined by 1)\ a N--connection structure, $\mathbf{N}:TE=hE\oplus vE$, and
2) a Lie aglebroid structure determined by 2a) a real \textit{vector bundle}
$\tau :\mathbf{E}\rightarrow M$ together with 2b) a \textit{Lie bracket} $%
\left\lfloor \cdot ,\cdot \right\rfloor $ on the spaces of global sections $%
Sec(\tau )$ \ of map $\tau $ and 2c) an \textit{anchor} map $\rho :\mathbf{E}%
\rightarrow TM,$ i.e., a bundle map over identity and constructed such that
for the homomorphism $\rho :Sec(\tau )\rightarrow \mathcal{X}(M)$ of $%
C^{\infty }(M)$--modules $\mathcal{X}$ this map satisfies the condition%
\begin{equation*}
\left\lfloor X,fY\right\rfloor =f\left\lfloor X,Y\right\rfloor +\rho
(X)(f)Y,~\forall X,Y\in Sec(\tau )\mbox{ and }f\in C^{\infty }(M).
\end{equation*}
\end{definition}

If $\mathbf{N}$ is integrable (i.e. with trivial N--connection structure), a
Lie d--algebroid is just a\textit{\ Lie algebroid} $\mathcal{E}%
=(E,\left\lfloor \cdot ,\cdot \right\rfloor ,\rho )$. For "non-boldface"
constructions, see details in Refs. \cite{higgins,mackenzie,vaisman,martinez1,martinez2,dleon1,cortes}. The anchor map $\rho $ is
equivalent to a homomorphism between the Lie algebras $\left( Sec(\tau
),\left\lfloor \cdot ,\cdot \right\rfloor \right) $ and $\left( \mathcal{X}%
(M),\left\lfloor \cdot ,\cdot \right\rfloor \right) .$\footnote{%
Locally, the properties of a Lie d--algebroid $\mathcal{E}$ are determined
by the functions $\rho _{a}^{i}(x^{k})$ and $C_{ab}^{e}(x^{k}),$ where $%
x=\{x^{k}\}$ are local coordinates on a chart $U$ on $M$ (or on $hV$ if a
nonholnomic manifold is considered)$,$ with $\rho (e_{a})=\rho _{a}^{i}(x)%
\mathbf{e}_{i}$ and $\left\lfloor e_{a},e_{b}\right\rfloor =\mathbf{C}%
_{ab}^{f}(x)e_{f},$ satisfying the following equations $\rho _{a}^{i}\mathbf{%
e}_{i}\rho _{b}^{j}-\rho _{b}^{i}\mathbf{e}_{i}\rho _{a}^{j}=\rho
_{f}^{j}C_{ab}^{f}$ and $\sum\limits_{\mbox{cycl\ }(a,b,f)}\left( \rho
_{a}^{i}\partial _{i}C_{be}^{f}+C_{be}^{d}C_{ad}^{f}\right) =0$. Boldface
operators are defined by N--coefficients similarly to (\ref{dbases}). For
trivial N--elongated partial derivatives and differentials, we can use local
coordinate frames when $\mathbf{e}_{i}\rightarrow \partial _{i},$ $%
e_{a}=\partial _{a},e^{i}=dx^{i},\mathbf{e}^{a}\rightarrow dy^{a}.$%
\par
The exterior differential on $\mathcal{E}$ with nonholonomic $\mathbf{E}$
can be defined in standard form with $E.$ We introduce on $\mathcal{E}$ the
operator $d:Sec(\bigwedge\nolimits^{k}\tau ^{\ast })\rightarrow
Sec(\bigwedge\nolimits^{k+1}\tau ^{\ast }),d^{2}=0,$ where $\bigwedge $ is
the antisymmetric product operator. The contributions of a N--connection can
be seen from such formulas for a smooth formula $f:M\rightarrow \mathbb{R}%
,df(X)=\rho (X)f,$ for $X\in Sec(\tau ),$ when $dx^{i}=\rho _{a}^{i}\mathbf{e%
}^{a}\mbox{ and }d\mathbf{e}^{f}=-\frac{1}{2}C_{ab}^{f}\mathbf{e}^{a}\wedge
\mathbf{e}^{b}$. With respect to any section $X$ on $M,$ we can define the
Lie derivative $\mathcal{L}_{X}=i_{X}\circ d+d\circ
i_{X}:~Sec(\bigwedge\nolimits^{k}\tau ^{\ast })\rightarrow
Sec(\bigwedge\nolimits^{k}\tau ^{\ast })$, using the cohomology operator $d$
and its inverse $i_{X}.$}

\begin{example}
\textbf{Nonholonomic Lie algebroids:} If $\mathbf{E=}TV$ for a nonolonomic
manifold $\mathbf{V}=(V,\mathbf{N)}$ with N--splitting (\ref{whitn}), the
values like $\mathcal{X}(M)$ are considered for $M\rightarrow hV$ and
sections $\tau $ are modelled on $vV.$ We shall construct Ricci soliton
configurations, or Einstein manifolds, with Lie algebroid symmetry in
section \ref{s4}.
\end{example}

Let us extend the concept of prolongation Lie algebroid \cite%
{martinez1,martinez2,dleon1,cortes} (in our case) in order to include
N--connections. We consider a Lie d--algebroid $\mathcal{E}=(\mathbf{E}%
,\left\lfloor \cdot ,\cdot \right\rfloor ,\rho )$ and a fibration $\pi :%
\mathbf{P}\rightarrow M$ \ both defined over the same manifold $M.$ In
general, $\mathbf{E,}$ $\mathbf{P}$ and $\mathbf{TM,}$ or $\mathbf{TV,}$ may
be enabled with different N--connection structures. The local coordinates
are $u^{\underline{\alpha }}=(x^{i},y^{A})\in P$ when $\{e_{a}\}$ will be
used for a local basis of sections of $\mathbf{E}.$ If $\mathbf{E=TV,}$ we
write $u^{\underline{\alpha }}=(x^{i},y^{I}).$ In our constructions, we can
consider that $\mathbf{P=E}.$ The anchor map $\rho :\mathbf{E}\rightarrow
\mathbf{TM}$ and the tangent map $\mathbf{T}\pi :\mathbf{TP\rightarrow TM}$
are all defined to be $h$--$v$--adapted for nonholonomic bundles and/or
fibered structures. These structures can be used to construct the subset
\begin{equation}
\mathcal{T}_{s}^{E}\mathbf{P}:=\{(b,v)\in \mathbf{E}_{x}\times T_{x}\mathbf{P%
};\rho (b)=T_{p}\pi (v);p\in \mathbf{P}_{x},\pi (p)=x\in M\}  \label{subset}
\end{equation}%
and prove this result:

\begin{theorem}
--\textbf{Definition.} \label{theordif} The the prolongation $\mathcal{T}^{%
\mathbf{E}}\mathbf{P}:=\bigcup_{s\in S}\mathcal{T}_{s}^{\mathbf{E}}\mathbf{P}
$ of a nonholonomic $\mathbf{E}$ over $\pi $ is another Lie d--algebroid
(the construction (\ref{subset}) can be considered for any set of charts
covering such spaces).
\end{theorem}

Similarly to holonomic configurations with trivial N--connection structure,
the prolongation Lie d--algebroid $\mathcal{T}^{\mathbf{E}}\mathbf{P}$ is
called the (nonholonomic) $\mathbf{E}$--tangent bundle to $\pi ,$ which is
also a nonholonomic vector bundle over $\mathbf{P.}$ The corresponding
projection $\tau _{\mathbf{P}}^{\mathbf{E}}$ is just onto the first factor, $%
\tau _{\mathbf{P}}^{\mathbf{E}}(b,v)=b$ being adapted to the N--connection
structures. The elements of $\mathcal{T}^{\mathbf{E}}\mathbf{P}
$ are parameterized by N--adapted triples $(p,b,v) \in \mathcal{T}^{\mathbf{E}}\mathbf{P}\rightarrow $ $%
(b,v)\in \mathcal{T}^{\mathbf{E}}\mathbf{P}$ if that will not result in
ambiguities. The N--adaped anchor $\rho ^{\pi }:\mathcal{T}^{\mathbf{E}}%
\mathbf{P}\rightarrow \mathcal{T}\mathbf{P}$ is given by maps $\rho ^{\pi }$
$(p,b,v)=v,$ i.e. projection onto the third factor when the $h$--components
\ transforms in other $h$--components etc. For more special cases, we can
define  the projection onto the second factor (i.e. a morphism of Lie
d--algebroids over $\pi ),$ $\mathcal{T}\pi :\mathcal{T}^{\mathbf{E}}\mathbf{%
P}\rightarrow \mathbf{E},$ when $\mathcal{T}\pi (p,b,v)=b.$ An element $%
(p_{1},b_{1},v_{1})\in \mathcal{T}^{\mathbf{E}}\mathbf{P}$ is vertical if $%
\mathcal{T}\pi $ $(p_{1},b_{1},v_{1})=b_{1}=0,$ i.e. such elements are of
type $(p,0,v)$ when $v$ is a $\pi $--vertical vector (tangent to $\mathbf{P}$
at point $p).$

To understand consequences of Theorem \ref{theordif} let us consider some
local constructions. In coefficient form, any element of a prolongation Lie
d--algebroid $\mathcal{T}^{\mathbf{E}}\mathbf{P}$ can be parameterized $%
\overline{z}=(p,b,v)\in \mathcal{T}^{\mathbf{E}}\mathbf{P},$ for $%
b=z^{a}e_{a}$ and $v=\rho _{a}^{i}z^{a}\mathbf{e}_{i}+v^{A}\partial _{A},$
for $\partial /\partial y^{A},$ can be decomposed $\overline{z}=z^{a}%
\mathcal{X}_{a}+v^{A}\mathcal{V}_{A}.$ The couple $\left( \mathcal{X}_{a},%
\mathcal{V}_{A}\right) ,$ with vertical $\mathcal{V}_{A},$ defines a local
basis of sections of $\mathcal{T}^{\mathbf{E}}\mathbf{P}.$ For such bases,
we can write $\mathcal{X}_{a}=\mathcal{X}_{a}(p)=\left( e_{a}(\pi (p)),\rho
_{a}^{i}\mathbf{e}_{i\mid p}\right) $ and $\mathcal{V}_{A}=\left( 0,\partial
_{A\mid p}\right) ,$ where partial derivatives and their N--elongations are
taken in a point $p\in S_{x}.$ It is also possible to elaborate on
prolongation Lie d--algebroids a N--adapted exterior differential calculus.
The anchor map $\rho ^{\pi }(Z)=\rho _{a}^{i}Z^{a}\mathbf{e}%
_{i}+V^{A}\partial _{A}$ is an N--elongated operator acting on sections $Z$
with associated decompositions of type $\overline{z}.$ The corresponding Lie
brackets are
\begin{equation*}
\left\lfloor \mathcal{X}_{a},\mathcal{X}_{b}\right\rfloor ^{\pi }=C_{ab}^{f}%
\mathcal{X}_{f},~\left\lfloor \mathcal{X}_{a},\mathcal{V}_{B}\right\rfloor
^{\pi }=0,~~\left\lfloor \mathcal{V}_{A},\mathcal{V}_{B}\right\rfloor ^{\pi
}=0.
\end{equation*}%
Denoting by $\left( \mathcal{X}^{a},\mathcal{V}^{B}\right) $ the dual bases
to $\left( \mathcal{X}_{a},\mathcal{V}_{A}\right) ,$ we can elaborate a
differential calculus for N--adapted differential forms using
\begin{equation}
dx^{i}=\rho _{a}^{i}\mathcal{X}^{a},\mbox{ for  }d\mathcal{X}^{f}=-\frac{1}{2%
}C_{ab}^{f}\mathcal{X}^{a}\wedge \mathcal{X}^{b},\mbox{ and \ }dy^{A}=%
\mathcal{V}^{A},\mbox{ \
for \ }d\mathcal{V}^{A}=0.  \label{extercalc}
\end{equation}

N--connections can be introduced on $\mathcal{T}^{\mathbf{E}}\mathbf{P}$
similarly to (\ref{whitn}).

\begin{definition}
On a prolongation Lie d--algebroid, a N--connection is defined by a $h$--$v$%
--splitting,
\begin{equation}
\mathcal{N}:\mathcal{T}^{\mathbf{E}}\mathbf{P}=h\mathcal{T}^{\mathbf{E}}%
\mathbf{P}\oplus v\mathcal{T}^{\mathbf{E}}\mathbf{P}.  \label{nonalg}
\end{equation}
\end{definition}

We can consider $\mathcal{N}:\mathcal{T}^{E}\mathbf{P\rightarrow }\mathcal{T}%
^{E}\mathbf{P,}$ with $\mathcal{N}^{2}=id,$ as a nonholonomic vector bundle,
and Lie d--algebroid, morphism defining an almost product structure on $\
^{P}\pi :\ T\mathbf{P}\rightarrow \mathbf{P},$ for a smooth map on $T\mathbf{%
P}\backslash \{0\},$ were $\{0\}$ denotes the set of null sections. Any
N--connection $\mathcal{N}$ induces $h$- and $v$--projectors for every
element $\overline{z}=(p,b,v)\in $ $\mathcal{T}^{\mathbf{E}}\mathbf{P},$
when $h(\overline{z})=~^{h}z$ and $v(z)=~^{v}z,$ for $h=\frac{1}{2}(id+%
\mathcal{N})$ and $v=\frac{1}{2}(id-\mathcal{N}).$ The respective $h$- and $%
v $--subspaces are $h\mathcal{T}^{\mathbf{E}}\mathbf{P}=\ker (id-\mathcal{N}%
) $ and $v\mathcal{T}^{\mathbf{E}}\mathbf{P}=\ker (id+\mathcal{N}).$

Let us analyze some local constructions related to $\mathcal{N}$--connection
structures for prolongation Lie d--algebroids. Locally, N--connections are
determined respectively by their coefficients $\mathbf{N=\{}N_{i}^{A}\}$ and
$\mathcal{N}\mathbf{=\{}\mathcal{N}_{a}^{A}\},$ when
\begin{equation}
\mathbf{N=}N_{i}^{A}(x^{k},y^{B})dx^{i}\otimes \partial _{A}\mbox{ \ and \ }%
\mathcal{N}=\mathcal{N}_{a}^{A}\mathcal{X}^{a}\otimes \mathcal{V}_{A}.
\label{nonlalg}
\end{equation}%
Such structures on $T\mathbf{P}$ and $\mathcal{T}^{\mathbf{E}}\mathbf{P}$
are compatible if $\mathcal{N}_{a}^{A}=N_{i}^{A}\rho _{a}^{i}.$ Using $%
\mathcal{N}_{a}^{A},$ we can generate sections $\delta _{a}:=\mathcal{X}_{a}-%
\mathcal{N}_{a}^{A}\mathcal{V}_{A}$ as a local basis of $h\mathcal{T}^{%
\mathbf{E}}\mathbf{P}.$

\begin{corollary}
Any N--connection $\mathcal{N}_{a}^{A}$ on $\mathcal{T}^{E}\mathbf{P}$
determines a N--adapted frame structure%
\begin{equation}
~\mathbf{e}_{\overline{\alpha }}:=\{\delta _{a}=\mathcal{X}_{a}-\mathcal{N}%
_{a}^{C}\mathcal{V}_{C},\mathcal{V}_{A}\},  \label{dderalg}
\end{equation}%
and its dual
\begin{equation}
\mathbf{e}^{\overline{\beta }}:=\{\mathcal{X}^{a},\delta ^{B}=\mathcal{V}%
^{B}+\mathcal{N}_{c}^{B}\mathcal{X}^{c}\}.  \label{ddifalg}
\end{equation}
\end{corollary}

\begin{proof}
In above formulas, the \textquotedblright overlined\textquotedblright\ small
Greek indices split $\overline{\alpha }=(a,A)$ if an arbitrary vector
bundles $\mathbf{P}$ is considered, or $\overline{\alpha }=(a,b)$ for $%
\mathbf{P=E}$. A proof follows from an explicit construction such N--adapted
frames. Then the formulas can be considered for arbitrary frames of
references. $\square $
\end{proof}

\vskip5pt

The N--adapted bases (\ref{dderalg}) satisfy the relations $\mathbf{e}_{%
\overline{\alpha }}~\mathbf{e}_{\overline{\beta }}-~\mathbf{e}_{\overline{%
\beta }}~\mathbf{e}_{\overline{\alpha }}=W_{\overline{\alpha }\overline{%
\beta }}^{\overline{\gamma }}\mathbf{e}_{\overline{\gamma }},$ with
nontrivial anholonomy coefficients $W_{\overline{\alpha }\overline{\beta }}^{%
\overline{\gamma }}=\{C_{ab}^{f},\Omega _{ab}^{C},\partial _{B}\mathcal{N}%
_{c}^{C}\}.$ The corresponding generalized Lie brackets are
\begin{equation}
\left\lfloor \delta _{a},\delta _{b}\right\rfloor ^{\pi }=C_{ab}^{f}\delta
_{f}+\Omega _{ab}^{C}\mathcal{V}_{C},~\left\lfloor \delta _{a},\mathcal{V}%
_{B}\right\rfloor ^{\pi }=(\partial _{B}\mathcal{N}_{a}^{C})\mathcal{V}%
_{C},~~\left\lfloor \mathcal{V}_{A},\mathcal{V}_{B}\right\rfloor ^{\pi }=0.
\label{extercalc1}
\end{equation}

\begin{definition}
The curvature of N--connection $\mathcal{N}_{a}^{A}$ is by definition the
Neig\-enhuis tensor $~^{h}N$ of the operator $h,$ {\small
\begin{equation*}
\ ^{h}N(\cdot ,\cdot )=\left\lfloor h\cdot ,h\cdot \right\rfloor ^{\pi
}-h\left\lfloor h\cdot ,\cdot \right\rfloor ^{\pi }-h\left\lfloor \cdot
,h\cdot \right\rfloor ^{\pi }+h^{2}\left\lfloor h\cdot ,h\cdot \right\rfloor
^{\pi }=-\frac{1}{2}\Omega _{ab}^{C}\mathcal{X}^{a}\wedge \mathcal{X}%
^{b}\otimes \mathcal{V}_{C},
\end{equation*}%
} where
\begin{equation}
\Omega _{ab}^{C}=\delta _{b}\mathcal{N}_{a}^{C}-\delta _{a}\mathcal{N}%
_{b}^{C}+C_{ab}^{f}\mathcal{N}_{f}^{C}.  \label{cncalgebr}
\end{equation}
\end{definition}

It should be noted that above formulas for $\mathcal{T}^{\mathbf{E}}\mathbf{P%
}$ mimic (on sections of $E$ for $\mathbf{P=E}$) the geometry of tangent
bundles and/or nonholonomic manifolds of even dimension, endowed with
N--connection structure. On applications in modern classical and quantum
gravity, with various modifications, and nonholonomic Ricci flow theory, see
Refs. \cite{vrflg,vricci1,vncricci}. If $\mathbf{P\neq E,}$ we model
nonholonomic vector bundle and generalized Riemann geometries on sections of
$\mathcal{T}^{\mathbf{E}}\mathbf{E.}$

\subsection{Canonical structures on Lie d--algebroids}

Almost K\"{a}hler Lie algebroid geometries can be modelled on prolongation
of Lie d--algebroids.

\subsubsection{d--connections and d--metrics on $\mathcal{T}^{\mathbf{E}}%
\mathbf{P}$}

\begin{definition}
\label{defdcon}A distinguished connection, d--connection, $\mathcal{D}=(h%
\mathcal{D},v\mathcal{D})$, on $\mathcal{T}^{\mathbf{E}}\mathbf{P}$ is a
linear connection preserving under parallelism the splitting (\ref{nonalg}).
\end{definition}

The formulas for distinguished torsion and curature can be generalized for
prolongation Lie d--algebroids.

\begin{definition}
The torsion and curvature for any d--connection $\mathcal{D}$ on $\mathcal{T}%
^{E}\mathbf{P}$ are defined respectively by
\begin{eqnarray*}
\mathcal{T}(\overline{x},\overline{y}) &:=&\mathcal{D}_{\overline{x}}%
\overline{y}-\mathcal{D}_{\overline{y}}\overline{x}+\left\lfloor \overline{x}%
,\overline{y}\right\rfloor ^{\pi },\mbox{ and } \\
\mathcal{R}(\overline{x},\overline{y})\overline{z} &:=&\left( \mathcal{D}_{%
\overline{x}}~\mathcal{D}_{\overline{y}}-~\mathcal{D}_{\overline{y}}\mathcal{%
D}_{\overline{x}}-\mathcal{D}_{\left\lfloor \overline{x},\overline{y}%
\right\rfloor ^{\pi }}\right) \overline{z}.
\end{eqnarray*}
\end{definition}

Let us consider sections $\overline{x},\overline{y},\overline{z}$ \ of $%
\mathcal{T}^{E}\mathbf{P}$ when, for instance $\overline{z}=z^{\overline{%
\alpha }}\mathbf{e}_{\overline{\alpha }}=z^{a}\delta _{a}+z^{A}\mathcal{V}%
_{A},$ or $\overline{z}=h\overline{z}+v\overline{z}.$ Using the rules of
absolute differentiation (\ref{extercalc}) for N--adapted bases $~\mathbf{e}%
_{\overline{\alpha }}:=\{\delta _{a},\mathcal{V}_{A}\}$ and $\mathbf{e}^{%
\overline{\beta }}:=\{\mathcal{X}^{\alpha },\delta ^{B}\}$ and associating
to $\mathcal{D}$ a d--connection 1--form $\mathbf{\Gamma }_{\ \overline{%
\alpha }}^{\overline{\gamma }}:=\mathbf{\Gamma }_{\ \overline{\alpha }%
\overline{\beta }}^{\overline{\gamma }}\mathbf{e}^{\overline{\beta }},$ we
can compute the N--adapted coefficients of the torsion $\mathcal{T}=\{%
\mathbf{T}_{\ \overline{\beta }\overline{\gamma }}^{\overline{\alpha }}\}$
and curvature $\mathcal{R}=\{\mathbf{R}_{\ \overline{\beta }\overline{\gamma
}\overline{\delta }}^{\overline{\alpha }}\}$ 2--forms, see details in
Appendix.

We can introduce a metric structure as a nondegenerate symmetric second rank
tensor $\overline{\mathbf{g}}=\{$ $\mathbf{g}_{\overline{\alpha }\overline{%
\beta }}\}\mathbf{.}$

\begin{proposition}
--\textbf{Definition.} Any metric structure on $\mathcal{T}^{E}\mathbf{P}$
can be represented in N--adapted form as a d--metric $\overline{\mathbf{g}}%
\mathbf{=}h\mathbf{g\oplus }v\mathbf{g}$.
\end{proposition}

\begin{proof}
It follows from an explicit construction adapted to a Whitney sum (\ref%
{nonalg}) and with respect to N--adapted frames,
see formula (\ref{dm}). $\square $
\end{proof}

\vskip5pt

If a d--metric $\overline{\mathbf{g}}$ and a d--connection $\mathcal{D}$ are
independent geometric structures, such values are characterized
(additionally to $\mathcal{T}$ and $\mathcal{R)}$ by a nonmetricity field $%
\mathcal{Q}(\overline{y}):=\mathcal{D}_{\overline{y}}\overline{\mathbf{g}},$
with N--adapted coefficients $\mathbf{Q}_{\ \overline{\alpha }\overline{%
\beta }}^{\overline{\gamma }}=\mathcal{D}^{\overline{\gamma }}\overline{%
\mathbf{g}}_{\overline{\alpha }\overline{\beta }}$.

\begin{proposition}
--\textbf{Definition.} The data $\left( \overline{\mathbf{g}},\mathcal{D}%
\right) $ are metric compatible if $\mathcal{Q}=\mathcal{D}\overline{\mathbf{%
g}}=0$ holds in N--adapted form for any h-- and v--components.
\end{proposition}

\begin{proof}
It follows from a straightforward computation when the coefficients of
d--metric $\mathbf{g}_{\overline{\alpha }\overline{\beta }}$ (\ref{dm}) are
introduced into $\mathcal{D}_{\overline{y}}\overline{\mathbf{g}}=0$, for $%
\overline{y}=y^{\overline{\alpha }}\mathbf{e}_{\overline{\alpha }%
}=y^{a}\delta _{a}+y^{A}\mathcal{V}_{A}.$ Such a condition splits into
respective conditions for $h$-$v$--components which with respect to
N--adapted frames are $\mathcal{D}_{f}\mathbf{g}_{ab}=0,\mathcal{D}_{A}%
\mathbf{g}_{ab}=0,\mathcal{D}_{f}\mathbf{g}_{AB}=0,\mathcal{D}_{C}\mathbf{g}%
_{AB}=0.$ $\square $
\end{proof}

\subsubsection{The canonical d--connection}

On $\mathcal{T}^{\mathbf{E}}\mathbf{P,}$ we can define two important metric
compatible linear connection structures completely determined by  $\mathbf{g}_{\overline{\alpha }%
\overline{\beta }}.$ The first one is the standard torsionless Levi--Civita
connection $\mathbf{\nabla }$ (which is not adapted to the N--connection
splitting) and the second one is a Lie d--algebroid generalization, $%
\widehat{\mathbf{D}}\rightarrow \widehat{\mathcal{D}}=h\widehat{\mathcal{D}}%
+v\widehat{\mathcal{D}},$ when the $h$-$v$--splitting is determined by (\ref%
{nonalg}).

\begin{theorem}
\label{thcdc}There is a canonical d--connection $\widehat{\mathcal{D}}$
completely defined by data $(\mathcal{N},\mathbf{g}_{\overline{\alpha }%
\overline{\beta }})$\ for which $\widehat{\mathcal{D}}\overline{\mathbf{g}}%
=0 $ and $h$- and $v$-torsions of $\widehat{\mathcal{T}}$ are prescribed,
respectively, to be with N--adapted coefficients $\widehat{T}_{\
bf}^{a}=C_{\ bf}^{a}$ and $\widehat{T}_{\ BC}^{A}=0.$
\end{theorem}

\begin{proof}
See proof in section \ref{assectcan}. $\square $
\end{proof}

\vskip5pt
For trivial algebroid structures, \ $\widehat{\mathcal{D}}\rightarrow
\widehat{\mathbf{D}}$ as on usual nonholonomic manifolds and/or vector
bundles.

\begin{remark}
\begin{itemize}
\item \label{remdistalg}There is a canonical distortion relation $\widehat{%
\mathcal{D}}=\overline{\mathbf{\nabla }}+\widehat{\mathcal{Z}},$ when both
linear connections $\widehat{\mathcal{D}}$ and $\overline{\mathbf{\nabla }}$
(the last one is the Levi--Civita connection) and the distorting d--tensor $%
\widehat{\mathcal{Z}}$ are defined by the same data $(\mathcal{N},\mathbf{g}%
_{\overline{\alpha }\overline{\beta }})$.  In explicit form, the N--adapted
coefficients of such values are computed following formulas (\ref{candcon})
and (\ref{distrel1}).

\item We note that $h\widehat{\mathcal{T}}^{\alpha }=0$ for $\widehat{%
\mathbf{D}}$ on $\mathbf{TM}$ but $h\widehat{\mathcal{T}}^{\alpha }\neq 0$
for $\widehat{\mathcal{D}}$ on $\mathcal{T}^{\mathbf{E}}\mathbf{P.}$ The
formulas for $\widehat{L}_{bf}^{a}$ \ in (\ref{candcon}) contain additional
terms with $C_{\ bf}^{a}$ which results in nontrivial $\widehat{T}_{\
bf}^{a}=C_{\ bf}^{a}$ and additional terms in N--adapted coefficients $%
\mathbf{R}_{\ ebf}^{a}$ and $\mathbf{R}_{\ Bbf}^{A}$ of curvature (\ref%
{dcurv}).
\end{itemize}
\end{remark}

In order to generate exact solutions, it may be more convenient to work with
an auxiliary d--connection $\ ^{c}\widehat{\mathcal{D}}:=\mathbf{\nabla }+\
^{c}\widehat{\mathcal{Z}}$ for which $h\ ^{c}\widehat{\mathcal{T}}^{\alpha
}=0$ and $v\ ^{c}\widehat{\mathcal{T}}^{\alpha }=0.$ In N--adapted form this
results in $\ ^{c}\widehat{T}_{\ bf}^{a}=0$ and $\ ^{c}\widehat{T}_{\
BC}^{A}=0$ but, in general, $\widehat{\mathcal{R}}\neq \ ^{c}\widehat{%
\mathcal{R}}$ and $\widehat{\mathcal{T}}\neq \ ^{c}\widehat{\mathcal{T}}.$
The non--trivial Lie d--algebroid structure is encoded in $\ ^{c}\widehat{%
\mathcal{Z}}$ via structure functions $\rho _{a}^{i}$ and $C_{\ bf}^{a}$ and
N--elongated frames (\ref{dderalg}) and (\ref{ddifalg}). The N--adapted
coefficients of $\ ^{c}\widehat{\mathcal{D}}$ are computed $\ ^{c}\widehat{L}%
_{bf}^{a}=\frac{1}{2}\mathbf{g}^{ae}\left( \delta _{f}\mathbf{g}_{be}+\delta
_{b}\mathbf{g}_{fe}-\delta _{e}\mathbf{g}_{bf}\right) $ and $\ ^{c}\widehat{L%
}_{Bf}^{A}=\widehat{L}_{Bf}^{A},\ ^{c}\widehat{B}_{\beta C}^{\alpha }=%
\widehat{B}_{\beta C}^{\alpha },\ \ ^{c}\widehat{B}_{BC}^{A}=\ \widehat{B}%
_{BC}^{A}$ are those from (\ref{candcon}).

\subsection{Almost symplectic geometric data}

\subsubsection{Semi--spray configurations and N--connections}

We show how a canonical almost symplectic structure can be generated on $%
\mathcal{T}^{\mathbf{E}}\mathbf{P}$ by any regular effective Lagrange
function $\mathcal{L}$.

\begin{lemma}
Prescribing any (effective) Lagrangian $\mathcal{L},$ we can construct a
canonical N--connection $~^{q}\mathcal{N}:=-\mathcal{L}i_{q}S$ defined by a
semi--spray $q=y^{a}\mathcal{X}_{a}+q^{A}\mathcal{V}_{A}$ and Lie derivative
$\mathcal{L}i_{q}$ acting on any $X\in Sec(T\mathbf{E})$ following formula $%
~^{q}\mathcal{N}(X)=-\left\lfloor q,SX\right\rfloor ^{\pi }+S\left\lfloor
q,X\right\rfloor ^{\pi }.$
\end{lemma}

\begin{proof}
We use the semi--spray formula \ $Sq=\bigtriangleup $ with the operators $S$
and $\bigtriangleup $ from (\ref{form1}) and compute
\begin{equation*}
^{q}\mathcal{N}(\mathcal{X}_{a})=-\left\lfloor q,S(\mathcal{X}%
_{a})\right\rfloor ^{\pi }+S\left\lfloor q,\mathcal{X}_{a}\right\rfloor
^{\pi }=\mathcal{X}_{a}+(\partial _{a}q^{b}+y^{f}C_{fa}^{b})\mathcal{V}_{b}.
\end{equation*}%
For $~^{q}\mathcal{N}(\mathcal{V}_{a})=-\mathcal{V}_{a}$ and $~^{q}\mathcal{N%
}(\mathcal{X}_{a})=\mathcal{X}_{a}-2~^{q}\mathcal{N}_{a}^{f}(x,y)\mathcal{V}%
_{f},$ we define the N--connection coefficients $\mathcal{N}_{\alpha }^{f}=-%
\frac{1}{2}(\partial _{a}q^{f}+y^{b}C_{ba}^{f})$, see formulas (\ref{nonlalg}%
). $\square $
\end{proof}

\vskip5pt

We can formulate on Lie d--aglebroids the analog of Theorem \ref{theoranalog}:

\begin{theorem}
Any effective regular Lagrangian $\mathcal{L}\in C^{\infty }(\mathbf{E})$
defines a canonical N--connection on prolongation Lie algebroid $\mathcal{T}%
^{\mathbf{E}}\mathbf{E,}$
\begin{equation}
\widetilde{\mathcal{N}}=\{\widetilde{\mathcal{N}}_{a}^{f}=-\frac{1}{2}%
(\partial _{a}\varphi ^{f}+y^{b}C_{ba}^{f})\},  \label{canonnc}
\end{equation}%
determined by semi--spray configurations encoding the solutions of the
Euler--Lagrange equations (\ref{eleq}).
\end{theorem}

\begin{proof}
It is a straightforward consequence of above Lemma and (\ref{semispray}). To
generate N--connections, we can use sections $\Gamma _{\mathcal{L}}=y^{a}%
\mathcal{X}_{a}+\varphi ^{a}\mathcal{V}_{a},$ with $q^{e}=\varphi
^{e}(x^{i},y^{b})$ (\ref{semispray}). Let us consider the value $S=y^{a}\rho
_{a}^{i}\frac{\partial \mathcal{L}}{\partial x^{i}}-2\tilde{G}^{A}\frac{%
\partial }{\partial y^{A}},$ where $\tilde{G}^{A}$ is associated to the
"nonlinear" geodesic equations for sections, $\frac{dy^{A}}{d\tau }+2\tilde{G%
}^{A}=0,$ depending on real parameter $\tau ,$ and define a N--connection
structure $\tilde{N}_{a}^{F}=\frac{\partial \tilde{G}^{F}}{\partial y^{a}}$
of type (\ref{nonlalg}). For $\mathbf{P=E,}$ such sections can be related to
the integral curves of the Euler--Lagrange equations (\ref{eleq}) if we
chose the sections $\tilde{G}^{F}\rightarrow \tilde{G}^{f}$ and $\varphi
^{f}(x^{i},y^{b})$ in such forms that $\widetilde{\mathcal{N}}_{a}^{f}=-%
\frac{1}{2}(\partial _{a}\varphi ^{f}+y^{b}C_{ba}^{f})=\frac{\partial \tilde{%
G}^{f}}{\partial y^{a}}.$ The constructions can be performed on any chart
covering such spaces, i.e. we can prove that the coefficients (\ref{canonnc}%
) define a N--connection structure (\ref{nonalg}). $\square $
\end{proof}

\vskip5pt

\begin{proposition}
Any metric structure on $\mathcal{T}^{\mathbf{E}}\mathbf{P}$ can be
represented in N--adapted form as a d--metric $\overline{\mathbf{g}}\mathbf{=%
}h\mathbf{\tilde{g}\oplus }v\mathbf{\tilde{g}}$ constructed as a formal
Sasaki lift determined by an effective regular generating function $\mathcal{%
L}$.
\end{proposition}

\begin{proof}
Sasaki lifts are used for extending certain metric structures from a base
manifold, for instance, to the total space of a tangent bundle, see details
in \cite{yano}. For vector bundles, the formulas (\ref{lfsm}) provide a
typical example of such a construction using canonical N--connection \
structure $\mathbf{\tilde{N}}.$ The method can be generalized for
prolongation Lie d--algebroids. At the first step, we use the canonical
N--connection $\widetilde{\mathcal{N}}=\{\widetilde{\mathcal{N}}_{a}^{f}\}$ (%
\ref{canonnc}) and construct N--adapted frames of type (\ref{dderalg}) and (%
\ref{ddifalg}), respectively,
\begin{equation}
~\widetilde{\mathbf{e}}_{\overline{\alpha }}:=\{\widetilde{\delta }_{a}=%
\mathcal{X}_{a}-\widetilde{\mathcal{N}}_{a}^{f}\mathcal{V}_{f},\mathcal{V}%
_{b}\}\mbox{ \ and
\ }\widetilde{\mathbf{e}}^{\overline{\beta }}:=\{\mathcal{X}^{a},\widetilde{%
\delta }^{b}=\mathcal{V}^{b}+\widetilde{\mathcal{N}}_{f}^{b}\mathcal{X}%
^{f}\}.  \label{ddcanad}
\end{equation}%
Then (the second step) we define a canonical d--metric
\begin{equation}
\widetilde{\mathbf{g}}:=\widetilde{\mathbf{g}}_{\overline{\alpha }\overline{%
\beta }}\mathbf{e}^{\overline{\beta }}\otimes \mathbf{e}^{\overline{\beta }}=%
\tilde{g}_{ab}\ \mathcal{X}^{a}\otimes \mathcal{X}^{b}+\ \tilde{g}_{ab}\
\widetilde{\delta }^{a}\otimes \widetilde{\delta }^{b}  \label{ldm}
\end{equation}%
using the Hessian (\ref{hessian}). Considering an arbitrary d--metric $%
\overline{\mathbf{g}}=\{\overline{\mathbf{g}}_{\overline{\alpha }^{\prime }%
\overline{\beta }^{\prime }}\}$ (\ref{dm}), we can find a regular $\mathcal{L%
}$ and certain frame transforms $\mathbf{e}_{\overline{\gamma }^{\prime
}}=e_{\ \overline{\gamma }^{\prime }}^{\overline{\gamma }}\mathbf{\tilde{e}}%
_{\overline{\gamma }}$ when $\overline{\mathbf{g}}_{\overline{\alpha }%
^{\prime }\overline{\beta }^{\prime }}=e_{\ \overline{\alpha }^{\prime }}^{%
\overline{\alpha }}e_{\ \overline{\beta }^{\prime }}^{\overline{\beta }}%
\mathbf{\tilde{g}}_{\overline{\alpha }\overline{\beta }}.$ We can work
equivalently both with $\overline{\mathbf{g}}$ and/or $\mathbf{\tilde{g}}$
if a nonholonomic distribution $\mathcal{L}$ is prescribed and the vielbein
coefficients $e_{\ \overline{\gamma }^{\prime }}^{\overline{\gamma }}$ are
defined as solutions of the corresponding algebraic quadratic system of
equations for some chosen data $\overline{\mathbf{g}}_{\overline{\alpha }%
^{\prime }\overline{\beta }^{\prime }}$ and $\mathbf{\tilde{g}}_{\overline{%
\alpha }\overline{\beta }}.$ $\square $
\end{proof}

\vskip5pt

\subsubsection{Riemann--Lagrange almost symplectic structures}

Let us consider canonical data defined respectively by N--elongated frames $%
\widetilde{\mathbf{e}}_{\overline{\alpha }}=(\widetilde{\mathbf{e}}%
_{a}=\delta _{a},\mathcal{V}_{A})$ (\ref{ddcanad}), N--connection $%
\widetilde{\mathcal{N}}$ (\ref{canonnc}) and d--metric $\overline{\mathbf{g}}%
=\widetilde{\mathbf{g}}$ (\ref{ldm}).

\begin{proposition}
--\textbf{Definition.} \label{propak}For any regular effective Lagrange
structure $\mathcal{L}$, we can define a canonical almost complex structure
on $\mathcal{T}^{\mathbf{E}}\mathbf{E}$ following formulas $\widetilde{%
\mathcal{J}}(\widetilde{\mathbf{e}}_{a})=-\mathcal{V}_{m+a}$\ and \ $%
\widetilde{\mathcal{J}}(\mathcal{V}_{m+a})=\widetilde{\mathbf{e}}_{a},$ when
$\widetilde{\mathcal{J}}\mathbf{\circ }\widetilde{\mathcal{J}}\mathbf{=-}%
\mathbb{I}$
\end{proposition}

\begin{proof}
It follows from an explicit construction of a d--tensor field%
\begin{equation*}
\widetilde{\mathcal{J}}\mathcal{=}\widetilde{\mathcal{J}}_{\ \overline{\beta
}}^{\overline{\alpha }}\widetilde{\mathbf{e}}_{\overline{\alpha }}\otimes
\widetilde{\mathbf{e}}^{\overline{\alpha }}=-\mathcal{V}_{m+a}\otimes
\mathcal{X}^{a}+\widetilde{\mathbf{e}}_{a}\otimes \widetilde{\delta }^{a}.
\end{equation*}%
Using vielbeins $e_{\ \overline{\alpha }^{\prime }}^{\overline{\alpha }}$
and their duals $e_{\overline{\beta }\ }^{\ \overline{\beta }^{\prime }},$
we can compute the coefficients of $\ \widetilde{\mathcal{J}}$ \ with
respect to any $\widetilde{e}_{\overline{\alpha }}$ and $\widetilde{e}^{%
\overline{\alpha }}$ on $\mathcal{T}^{\mathbf{E}}\mathbf{E},$ when $%
\widetilde{\mathcal{J}}_{\ \overline{\beta }}^{\overline{\alpha }}=e_{\
\overline{\alpha }^{\prime }}^{\overline{\alpha }}e_{\overline{\beta }\ }^{\
\overline{\beta }^{\prime }}\widetilde{\mathcal{J}}_{\ \overline{\beta }%
^{\prime }}^{\overline{\alpha }^{\prime }}. \square $
\end{proof}

\vskip5pt

In general, we can define an almost complex structure $\mathcal{J}$ on $%
\mathcal{T}^{\mathbf{E}}\mathbf{E}$ for an arbitrary N--connection $\mathcal{%
N}$ (\ref{nonalg}) by using N--adapted bases (\ref{dderalg}) and (\ref%
{ddifalg}) which are not necessarily induced by an effective Lagrange
function $\mathcal{L}$. This allows us to generate almost Hermitian models
and not almost K\"{a}hler ones.

\begin{definition}
The Nijenhuis tensor field for any almost complex structure $\widetilde{%
\mathcal{J}}$ on $\mathcal{T}^{\mathbf{E}}\mathbf{E}$ determined by a
N--connection $\mathcal{N}$ (equivalently, the curvature of N--connection $%
\mathcal{N}$) is by definition
\begin{equation}
\ ^{\mathcal{J}}\mathbf{\Omega }(\overline{x},\overline{y}):=\mathbf{-}[%
\overline{x},\overline{y}]\mathbf{+}[\mathcal{J}\overline{x},\mathcal{J}%
\overline{y}]\mathbf{-}\mathcal{J}[\mathcal{J}\overline{x},\overline{y}]%
\mathbf{-}\mathcal{J}[\overline{x},\mathcal{J}\overline{y}],  \label{neijt}
\end{equation}%
for any sections $\overline{x},\overline{y}$ \ of $\mathcal{T}^{\mathbf{E}}%
\mathbf{E.}$
\end{definition}

We can introduce an arbitrary almost symplectic structure as a 2--form on a
prolongation Lie d--algebroid.

\begin{definition}
An almost symplectic structure on $\mathcal{T}^{\mathbf{E}}\mathbf{P}$ is
defined by a nondegenerate 2--form $\theta =\frac{1}{2}\theta _{\overline{%
\alpha }\overline{\beta }}(x^{i},y^{B})e^{\overline{\alpha }}\wedge e^{%
\overline{\beta }}.$
\end{definition}

Using frame transforms, we can prove

\begin{proposition}
\label{prop02}For any $\theta $ on $\mathcal{T}^{\mathbf{E}}\mathbf{P}$ when
$h\theta (\overline{x}\mathbf{,}\overline{y}):=\theta (h\overline{x}\mathbf{,%
}h\overline{y}),v\theta (\overline{x}\mathbf{,}\overline{y}):=\theta (v%
\overline{x}\mathbf{,}v\overline{y})$, there is a unique N--connection $%
\mathcal{N}=\{\mathcal{N}_{a}^{A}\}$ (\ref{nonlalg}) when
\begin{equation}
\theta =(h\overline{x},v\overline{y})=0\mbox{ and }\theta =h\theta +v\theta .
\label{aux02}
\end{equation}
\end{proposition}

\begin{proof}
In N--adapted form,
\begin{equation}
\theta =\frac{1}{2}\theta _{ab}(x^{i},y^{C})\mathcal{X}^{a}\wedge \mathcal{X}%
^{b}+\frac{1}{2}\theta _{AB}(x^{i},y^{C})\delta ^{A}\wedge \delta ^{B},
\label{aux03}
\end{equation}%
where the first term is for $h\theta $ and the second term is $v\theta ,$
i.e. we get the second formula in (\ref{aux02}). $\square $
\end{proof}

\begin{definition}
\begin{itemize}
\item a) An almost Hermitian model of a prolongation Lie d--algebroid $%
\mathcal{T}^{\mathbf{E}}\mathbf{E}$ equipped with a N--connection structure $%
\mathcal{N}$ is defined by a triple $\mathcal{H}^{\mathbf{E}}\mathbf{E}=(%
\mathcal{T}^{\mathbf{E}}\mathbf{E},\theta ,\mathcal{J}),$ where $\theta (%
\overline{x}\mathbf{,}\overline{y})\mathbf{:=g}(\mathcal{J}\overline{x}%
\mathbf{,}\overline{y}).$

\item b) A Hermitian prolongation Lie d--algebroid $\mathcal{H}^{\mathbf{E}}%
\mathbf{E}$ is almost K\"{a}hler, denoted $\mathcal{K}^{\mathbf{E}}\mathbf{E}%
,$ if and only if $d\mathbf{\theta }=0.$
\end{itemize}
\end{definition}

For effective regular Lagrange configurations, we can formulate:

\begin{theorem}
\label{thmr2}Having chosen a generating function $\mathcal{L},$ we can model
equivalently a prolongation Lie d--algebroid $\mathcal{T}^{\mathbf{E}}%
\mathbf{E}$ as an almost K\"{a}hler geometry, i.e. $\mathcal{H}^{\mathbf{E}}%
\mathbf{E}=\mathcal{K}^{\mathbf{E}}\mathbf{E}.$
\end{theorem}

\begin{proof}
For the canonical geometric data $(\overline{\mathbf{g}}=\mathbf{\tilde{g},}%
\widetilde{\mathcal{N}},\widetilde{\mathcal{J}}),$ we define the symplectic
form $\tilde{\theta}(\overline{x}\mathbf{,}\overline{y})\mathbf{:=\tilde{g}}(%
\mathcal{J}\overline{x}\mathbf{,}\overline{y})$ for any sections $\overline{x%
},\overline{y}$ \ of $\mathcal{T}^{\mathbf{E}}\mathbf{E.}$ In local
N--adapted form, $\tilde{\theta}=\ \tilde{g}_{ab}\delta ^{a}\wedge \mathcal{X%
}^{b}.$ Let us consider the form $\tilde{\omega}:=\frac{1}{2}\frac{\partial
\mathcal{L}}{\partial y^{m+a}}\mathcal{X}^{a}.$ Using Proposition \ref%
{prop02} and N--connection $\widetilde{\mathcal{N}}$ (\ref{canonnc}), we
prove that $\tilde{\theta}=d\tilde{\omega},$ which means that $d\tilde{\theta%
}=dd\tilde{\omega}=0.$ The constructions can be redefined in arbitrary
frames, $\theta _{\overline{\alpha }^{\prime }\overline{\beta }^{\prime
}}=e_{\ \overline{\alpha }^{\prime }}^{\overline{\alpha }}e_{\ \overline{%
\beta }^{\prime }}^{\overline{\beta }}\tilde{\theta}_{\overline{\alpha }%
\overline{\beta }},$ for a 2--form of type (\ref{aux03}),
\begin{equation}
\tilde{\theta}=\frac{1}{2}\tilde{\theta}_{ab}(x^{i},y^{C})\mathcal{X}%
^{a}\wedge \mathcal{X}^{b}+\frac{1}{2}\tilde{\theta}_{AB}(x^{i},y^{C})\tilde{%
\delta}^{A}\wedge \tilde{\delta}^{B}.  \label{canasf}
\end{equation}

$\square $
\end{proof}

\subsubsection{N--adapted symplectic connections}

Let us consider how d--connection structures can be defined on $\mathcal{H}^{%
\mathbf{E}}\mathbf{E}$ and/or $\mathcal{K}^{\mathbf{E}}\mathbf{E.}$

\begin{definition}
\label{defasstra}An almost symplectic d--connection $\ ^{\theta }\mathcal{D}$
for a model $\mathcal{H}^{\mathbf{E}}\mathbf{E}$ of $\mathcal{T}^{\mathbf{E}}%
\mathbf{E},$ or (equivalently) a d--connection compatible with an almost
symplectic structure $\theta ,$ is defined such that this linear connection
is N--adapted, i.e. a d--connection, and $\ ^{\theta }\mathcal{D}_{\overline{%
x}}\theta =0,$ for any section $\overline{x}$\textbf{\ }of $\mathcal{T}^{%
\mathbf{E}}\mathbf{E.}$
\end{definition}

\begin{lemma}
We can always fix a d--connection $\ ^{\circ }\mathcal{D}$ on $\mathcal{T}^{%
\mathbf{E}}\mathbf{E}$ and then construct an almost symplectic$\ ^{\theta }%
\mathcal{D}\mathbf{.}$
\end{lemma}

\begin{proof}
Let us consider a $\theta $ in N--adapted form (\ref{aux03}). Introducing%
\begin{eqnarray*}
\ ^{\circ }\mathcal{D} &=&\left\{ h\ \ ^{\circ }\mathcal{D}=(\ _{h}^{\circ }%
\mathcal{D}_{a},\ _{v}^{\circ }\mathcal{D}_{a});\ v\ ^{\circ }\mathcal{D}=(\
_{h}^{\circ }\mathcal{D}_{A},\ _{v}^{\circ }\mathcal{D}_{A})\right\} \\
&=&\{\ ^{\circ }\mathbf{\Gamma }_{\overline{\beta }\overline{\gamma }}^{%
\overline{\alpha }}=(\ ^{\circ }\mathbf{L}_{bf}^{a},\ ^{\circ }\mathbf{L}%
_{Bf\;}^{A};\ ^{\circ }\mathbf{B}_{bC}^{a},\ ^{\circ }\mathbf{B}_{BC}^{A})\},
\end{eqnarray*}%
we can verify that%
\begin{eqnarray*}
\ ^{\theta }\mathcal{D} &=&\left\{ h\ \ ^{\theta }\mathcal{D}=(\
_{h}^{\theta }\mathcal{D}_{a},\ _{v}^{\theta }\mathcal{D}_{a});\ v\ ^{\theta
}\mathcal{D}=(\ _{h}^{\theta }\mathcal{D}_{A},\ _{v}^{\theta }\mathcal{D}%
_{A})\right\} \\
&=&\{\ ^{\theta }\mathbf{\Gamma }_{\overline{\beta }\overline{\gamma }}^{%
\overline{\alpha }}=(\ ^{\theta }\mathbf{L}_{bf}^{a},\ ^{\theta }\mathbf{L}%
_{Bf\;}^{A};\ ^{\theta }\mathbf{B}_{bC}^{a},\ ^{\theta }\mathbf{B}%
_{BC}^{A})\}, \mbox{ with }
\end{eqnarray*}%
\begin{eqnarray}
\ ^{\theta }\mathbf{L}_{bf}^{a} &=&\ ^{\circ }\mathbf{L}_{bf}^{a}+\frac{1}{2}%
\theta ^{ae}\ _{h}^{\circ }\mathcal{D}_{f}\theta _{be},\ \ ^{\theta }\mathbf{%
L}_{Bf\;}^{A}=\ ^{\circ }\mathbf{L}_{Bf\;}^{A}\ +\frac{1}{2}\theta ^{AE}\
_{v}^{\circ }\mathcal{D}_{f}\theta _{EB},  \label{csdc} \\
\ ^{\theta }\mathbf{B}_{bC}^{a} &=&\ ^{\circ }\mathbf{B}_{bC}^{a}+\frac{1}{2}%
\theta ^{ae}\ _{h}^{\circ }\mathcal{D}_{C}\theta _{be},\ \ ^{\theta }\mathbf{%
B}_{BC}^{A}=\ ^{\theta }\mathbf{B}_{BC}^{A}+\frac{1}{2}\theta ^{AE}\
_{v}^{\circ }\mathcal{D}_{C}\theta _{EB},  \notag
\end{eqnarray}%
satisfies the conditions $\ _{h}^{\theta }\mathcal{D}_{a}\theta _{be}=0,\
_{v}^{\theta }\mathcal{D}_{a}\theta _{AB}=0,\ _{h}^{\theta }\mathcal{D}%
_{A}\theta _{be}=0,\ _{v}^{\theta }\mathcal{D}_{A}\theta _{AB}=0,$ which are
h- and v--projections of $\ ^{\theta }\mathcal{D}_{\overline{x}}\theta =0$
from Definition \ref{defasstr}. $\square $
\end{proof}

\vskip5pt

Let us introduce the operators
\begin{equation}
\Theta _{cd}^{ab}=\frac{1}{2}(\delta _{c}^{a}\delta _{d}^{b}-\theta
_{cd}\theta ^{ab})\mbox{ and }\Theta _{CD}^{AB}=\frac{1}{2}(\delta
_{C}^{A}\delta _{D}^{B}-\theta _{CD}\theta ^{AB}).  \label{thop}
\end{equation}

\begin{theorem}
\label{setscdc}The set of d--connections $\ ^{s}\mathbf{\Gamma }_{\overline{%
\beta }\overline{\gamma }}^{\overline{\alpha }}=(\ ^{s}\mathbf{L}_{bf}^{a},\
^{s}\mathbf{L}_{Bf}^{A};\ ^{s}\mathbf{B}_{bC}^{a},\ ^{s}\mathbf{B}_{BC}^{A})$
labeled by an abstract left index "$s$", compatible with an almost
symplectic structure $\theta $ (\ref{aux03}), are parameterized by
\begin{eqnarray}
\ ^{s}\mathbf{L}_{bc}^{a} &=&\ ^{\theta }\mathbf{L}_{bc}^{a}+\Theta
_{be}^{da}\ \mathbf{Y}_{dc}^{e},\ ^{s}\mathbf{L}_{Bc}^{A}=\ ^{\theta }%
\mathbf{L}_{Bc}^{A}+\Theta _{BE}^{CA}\ \mathbf{Y}_{Cc}^{E},  \label{fsdc} \\
\ ^{s}\mathbf{B}_{bC}^{a} &=&\ ^{\theta }\mathbf{B}_{bC}^{a}+\Theta
_{bf}^{ea}\ \mathbf{Y}_{eC}^{f},\ \ ^{s}\mathbf{B}_{BC}^{A}=\ ^{\theta }%
\mathbf{B}_{BC}^{A}+\Theta _{BF}^{EA}\ \mathbf{Y}_{EC}^{F},\   \notag
\end{eqnarray}%
where the N--adapted coefficients are given by (\ref{csdc}), the $\Theta $%
--operators are those from (\ref{thop}) and $\mathbf{Y}_{\overline{\beta }%
\overline{\gamma }}^{\overline{\alpha }}=(\mathbf{Y}_{dc}^{e},\mathbf{Y}%
_{Cc}^{E},\mathbf{Y}_{eC}^{f},\mathbf{Y}_{EC}^{F})$ are arbitrary d--tensor
fields.
\end{theorem}

\begin{proof}
It follows from straightforward N--adapted computations. $\square $
\end{proof}

\begin{remark}
The Lie algebroid structure functions $C_{bf}^{d}$ in (\ref{candcon}) can be
considered as an example of d--tensor fields $\mathbf{Y}_{\overline{\beta }%
\overline{\gamma }}^{\overline{\alpha }}$ in (\ref{fsdc}). On $\mathcal{T}^{%
\mathbf{E}}\mathbf{E,}$ we can work as on $\mathbf{TM}$ but for differnet
classes of nonholonomic distributions for sections. The d--connections $%
\widehat{\mathcal{D}},\ ^{\theta }\mathcal{D}$ can be constructed for
correspondingly defined N--connection structures $\mathcal{N},\widetilde{%
\mathcal{N}}$ when the main geometric properties are similar to some
geometric models with $\widehat{\mathbf{D}},\ ^{\theta }\mathbf{D}$ and
certain $\mathbf{N},\widetilde{\mathbf{N}}.$ The nonholonomic frame
structures on Lie d--algebroids are different from those on nonholonomic
tangent bundles because in the first case the vierbein fields encode the
ancor and Lie type structure functions.
\end{remark}

We can select a subclass of metric and/or almost symplectic compatible
d--connections on $\mathcal{T}^{\mathbf{E}}\mathbf{E}$ which are completely
defined by $\mathbf{g}$ and prescribed by an effective Lagrange structure $%
\mathcal{L}(x,y).$

\begin{theorem}
On $\mathcal{T}^{\mathbf{E}}\mathbf{E,}$ there is a unique normal
d--connection $\ ^{n}\mathcal{D} =$
\begin{eqnarray}
&& \{ h\ ^{n}\mathcal{D}=(\ _{h}^{n}\mathcal{D}_{a}=\widehat{\mathcal{D}}%
_{a},\ _{v}^{n}\mathcal{D}_{a}=\widehat{\mathcal{D}}_{a});v\ ^{n}\mathcal{D}%
=(\ _{h}^{n}\mathcal{D}_{A}=\widehat{\mathcal{D}}_{A},\ _{v}^{n}\mathcal{D}%
_{A}=\widehat{\mathcal{D}}_{A}) \}  \notag \\
&& = \{\ ^{n}\mathbf{\Gamma }_{\overline{\beta }\overline{\gamma }}^{%
\overline{\alpha }}=(\ ^{n}\mathbf{L}_{bf}^{a}=\widehat{\mathbf{L}}%
_{bf}^{a}, \ ^{n}\mathbf{L}_{B=m+b\ f\;}^{A=m+a}=\widehat{\mathbf{L}}%
_{bf}^{a};  \notag \\
&& \quad \ ^{n}\mathbf{B}_{b=B-m\ C}^{a=A-m}=\widehat{\mathbf{B}}_{BC}^{A},\
^{n}\mathbf{B}_{BC}^{A}=\widehat{\mathbf{B}}_{BC}^{A})\},  \label{ndc}
\end{eqnarray}%
which is metric compatible, $\widehat{\mathcal{D}}_{a}\mathbf{\tilde{g}}%
_{bc}=0$ and $\widehat{\mathcal{D}}_{A}\mathbf{\tilde{g}}_{BC}=0,$ and
completely defined by $\overline{\mathbf{g}}=\mathbf{\tilde{g}}$ and a fixed
$\mathcal{L}(x,y).$
\end{theorem}

\begin{proof}
We provide a proof constructing such a normal d--connection in explicit form
an example when $\ ^{n}\mathcal{D}=\widetilde{\mathcal{D}}$ generalizes the
concept of Cartan d--connection in Lagrange--Finsler geometry on $\mathbf{TM}
$. Such a d--connection is completely defined by couples of h-- and
v--components $\widetilde{\mathcal{D}}=(\widetilde{\mathcal{D}}_{a},%
\widetilde{\mathcal{D}}_{A}),$ i.e. $\ \widetilde{\mathbf{\Gamma }}_{%
\overline{\beta }\overline{\gamma }}^{\overline{\alpha }}=(\widehat{\mathbf{L%
}}_{bf}^{a},\widehat{\mathbf{B}}_{BC}^{A}).$ Let us chose {\small
\begin{equation}
\widehat{\mathbf{L}}_{bf}^{a}=\frac{1}{2}\widetilde{\mathbf{g}}^{ae}\
(\delta _{f}\widetilde{\mathbf{g}}_{be}+\delta _{b}\widetilde{\mathbf{g}}%
_{fe}-\delta _{e}\widetilde{\mathbf{g}}_{bf}),\ \widehat{\mathbf{B}}%
_{BC}^{A}=\frac{1}{2}\widetilde{\mathbf{g}}^{AD}\ (\mathcal{V}_{C}\widetilde{%
\mathbf{g}}_{BD}+\mathcal{V}_{B}\widetilde{\mathbf{g}}_{CD}-\mathcal{V}_{D}%
\widetilde{\mathbf{g}}_{BC}),  \label{cdcc}
\end{equation}%
} where the N--elongated derivatives are taken in the form (\ref{ddcanad})
and $\widetilde{\mathbf{g}}_{ab}=\widetilde{\mathbf{g}}_{A=m+a\ B=m+b}$ are
generated by canonical values using the Hessian (\ref{hessian}) and (\ref%
{ldm}) induced by a regular $\mathcal{L}(x,y),$ we can prove that this
d--connection is unique and satisfies the conditions of the theorem. Via
frame transforms, we can consider any metric structure $\overline{\mathbf{g}}%
\sim \mathbf{\tilde{g}.}$ $\square $
\end{proof}

\vskip5pt

Finally, we formulate this very important for our purposes result:

\begin{theorem}
The normal d--connection $\ ^{n}\mathcal{D}=\widetilde{\mathcal{D}}$ defines
a unique almost symplectic d--connection, $\widetilde{\mathcal{D}}\equiv \
^{\theta }\widetilde{\mathcal{D}},$ see Definition \ref{defasstra}, which is
N--adapted and compatible to the canonical almost symplecti form $\tilde{%
\theta}$ (\ref{canasf}), i.e. $\ ^{\theta }\widetilde{\mathcal{D}}\tilde{%
\theta}\mathbf{=}0$ and $\widetilde{\mathbf{T}}_{\ cb}^{a}=\widetilde{%
\mathbf{T}}_{\ CB}^{A}=0,$ see torsion coefficients (\ref{cdtors}).
\end{theorem}

\begin{proof}
Using the coefficients (\ref{cdcc}), we can check that such a normal
d--connection satisfies the conditions of this theorem. $\square $
\end{proof}

\vskip3pt

\begin{conclusion}
\label{cequiv}Prescribing an effective generating function $\mathcal{L}$ on $%
\mathcal{T}^{\mathbf{E}}\mathbf{E,}$ we can transform this prolongation Lie
d--algebroid into a canonical almost K\"{a}hler one, $\mathcal{K}^{\mathbf{E}%
}\mathbf{E}.$ It is possible to work equivalently with any geometric data
\begin{equation*}
\left[ \overline{\mathbf{g}},\mathcal{N},\widehat{\mathcal{D}}=\overline{%
\mathbf{\nabla }}+\widehat{\mathcal{Z}}\right] \approx \left[ \mathbf{\tilde{%
g},}\mathcal{L},\widetilde{\mathcal{N}},\widetilde{\mathcal{D}}\right]
\approx \left[ \tilde{\theta}(\cdot ,\cdot )\mathbf{:=\tilde{g}}(\widetilde{%
\mathcal{J}}\cdot \mathbf{,}\cdot ),\ ^{\theta }\widetilde{\mathcal{D}}%
\right] .
\end{equation*}%
The Lie algebroid structure functions $(\rho _{a}^{i},C_{ab}^{f})$ are
encoded into nonholonomic distributions on $\mathcal{T}^{\mathbf{E}}\mathbf{E%
}$ determining such equivalent prolongation Lie d--algebroid configurations.
\end{conclusion}

\subsection{Almost K\"{a}hler Einstein and Lagrange Lie d--algebroids}

We can formulate analogs of Einstein equations for different classes of
d--connections on prolongation Lie d--algebroids and almost K\"{a}hler
models.

\begin{corollary}
--\textbf{Definition} \label{corolricci} The Ricci tensor of a d--connection
$\mathcal{D}$ on a $\mathcal{T}^{\mathbf{E}}\mathbf{P}$ endowed with
d--metric structure $\overline{\mathbf{g}}$ is defined following formula $%
\mathcal{R}ic=\{\mathbf{R}_{\overline{\alpha }\overline{\beta }}:=\mathbf{R}%
_{\ \overline{\alpha }\overline{\beta }\overline{\gamma }}^{\overline{\gamma
}}\},$ see the coefficients (\ref{dcurv}) for Riemannian d--tensor $\mathcal{%
R}_{~\overline{\beta }}^{\overline{\alpha }}=\{\mathbf{R}_{\ \overline{\beta
}\overline{\gamma }\overline{\delta }}^{\overline{\alpha }}\}$, and
characterized by N--adapted coefficients{\small
\begin{equation}
\mathbf{R}_{\overline{\alpha }\overline{\beta }}=\{\mathbf{R}_{ab}:=\mathbf{R%
}_{\ abc}^{c},\ \mathbf{R}_{aA}:=-\mathbf{R}_{\ ~acA}^{c},\ \mathbf{R}_{Aa}:=%
\mathbf{R}_{\ ~AaB}^{B},\ \mathbf{R}_{AB}:=\mathbf{R}_{\ ~ABC}^{C}\}.
\label{driccialg}
\end{equation}%
}
\end{corollary}

\begin{proof}
The formulas for $h$--$v$--components (\ref{driccialg}) are respective
contractions of the coefficients (\ref{dcurv}). $\square $
\end{proof}

\vskip5pt

The scalar curvature $\ ^{s}\mathbf{R}$ of $\ \mathcal{D}$ is by definition
\begin{equation}
\ ^{s}\mathbf{R}:=\mathbf{g}^{\overline{\alpha }\overline{\beta }}\mathbf{R}%
_{\overline{\alpha }\overline{\beta }}=\mathbf{g}^{ab}\mathbf{R}_{ab}+%
\mathbf{g}^{AB}\mathbf{R}_{AB}.  \label{sdcurv}
\end{equation}%
Using (\ref{driccialg}) and (\ref{sdcurv}), we compute the Einstein
d--tensor $\mathbf{E}_{\overline{\alpha }\overline{\beta }}:=\mathbf{R}_{%
\overline{\alpha }\overline{\beta }}-\frac{1}{2}\mathbf{g}_{\overline{\alpha
}\overline{\beta }}\ \ ^{s}\mathbf{R}$ of $\mathcal{D}.$ Such a tensor can
be used for modeling effective gravity theories on sections of $\mathcal{T}$
$^{\mathbf{E}}\mathbf{P}$ with nonholonomic frame structure \cite%
{vcliffalg,vhamlagralg}.

\section{Almost K\"{a}hler -- Ricci Lie Algebroid Evolution}

\label{s3}

The goal of this section is to prove that N--adapted Ricci flow theories for
almost K\"{a}hler models of prolongation Lie algebroids, $\mathcal{K}^{%
\mathbf{E}}\mathbf{E,}$ can be formulated as models of generalized gradient
nonholonomic flows. We extend the Grisha Perelman's geometric thermodynamic
functional approach \cite{gper1,gper2,gper3} and show how modified R.
Hamilton type equations \cite{ham1,ham2} can be derived for the almost K\"{a}%
hler evolution of Lie d--algebroids.

Symplectic and almost symplectic geometric flows have been studied in modern
Ricci flow theory, see a series of examples and reviews of results in \cite%
{kleiner,rbook,chau,streets}. Our approach is very different from those with
"pure" complex and/or symplectic forms and connections.

\subsection{Perelman's functionals in almost K\"{a}hler variables}

There are both conceptual and technical difficulties which do not allow us
to formulate a generalized Ricci flow theory for non--Riemannian geometries
with independent metric and linear connection structures, or their almost
symplectic analogs. Nevertheless, unified geometric evolution theories can
be constructed for certain classes of nonholonomic manifolds and
Lagrange--Finsler spaces \cite{vricci1,vricci2,vfinsler} both with metric
compatible and noncompatible N--adapted connections if the fundamental
geometric objects are determined in unique forms by distortion relations
completely determined by a metric and/or almost symplectic structure.
\begin{remark}
\label{remarkakdef} Proofs of theorems for almost K\"{a}hler models
on prolongation Lie d--algebroids on $\mathcal{T}$ $^{\mathbf{E}}\mathbf{E}$
and/or $\mathcal{K}^{\mathbf{E}}\mathbf{E}$ can be obtained by N--adapted nonholonomic
deformations of geometric constructions for the
standard Levi--Civita connection $\overline{\mathbf{\nabla }},$ when
 $\left[ \overline{\mathbf{g}}\sim \mathbf{\tilde{g},}\mathcal{L},\widetilde{%
\mathcal{N}},\widetilde{\mathcal{D}}\right] \approx \left[ \tilde{\theta}%
(\cdot ,\cdot )\mathbf{:=\tilde{g}}(\widetilde{\mathcal{J}}\cdot \mathbf{,}%
\cdot ),\ ^{\theta }\widetilde{\mathcal{D}}=\overline{\mathbf{\nabla }}+%
\widetilde{\mathcal{Z}}\right]$.
The functions determining nonholonomic distributions, N--connection
coefficients and Lie algebroid structure functions are considered to be of
the same smooth class as the coefficients of $\overline{\mathbf{g}}$ and $%
\overline{\mathbf{\nabla }}.$
\end{remark}

The theory of Lagrange--Ricci flows on $\mathcal{T}^{\mathbf{E}}\mathbf{E}$
is formulated for evolving nonholonomic dynamical systems on the space of
equivalent geometric data $\left( \mathcal{L}:\overline{\mathbf{g}},%
\overline{\nabla }\right) $, when the functionals $\ _{\shortmid }\mathcal{F}
$ and $\ _{\shortmid }\mathcal{W}$ are postulated to be of Lyapunov type,
see below formulas (\ref{2pfrs1}) and (\ref{2pfrs2}). Ricci solitonic
configurations are defined as \textquotedblright fixed\textquotedblright\ on
$\tau $ points of the corresponding dynamical systems. The "stationary"
variational conditions depend on what type of the Ricci tensor we use, for
instance, that one for the connections $\overline{\nabla }$ or $\widehat{%
\mathcal{D}}.$ We can elaborate N--adapted almost K\"{a}hler scenarios if
the Perelman's functionals are re--defined in terms of geometric data $(%
\widetilde{\mathbf{g}},\widetilde{\mathcal{D}})$ and the derived flow
equations are considered in N--adapted variables. Both approaches are
equivalent if the distortion relations are considered for the same family of
metrics, $\overline{\mathbf{g}}(\tau )=\widetilde{\mathbf{g}}(\tau )$
correspondingly computed for a set $\mathcal{L}(\tau ).$

\begin{lemma}
For the scalar curvature and Ricci tensor\ determined by the distortion
relation
\begin{equation}
\widetilde{\mathcal{D}}=\ ^{\theta }\widetilde{\mathcal{D}}=\overline{%
\mathbf{\nabla }}+\widetilde{\mathcal{Z}},  \label{akdralg}
\end{equation}%
the Perelman's functionals are defined equivalently in almost K\"{a}hler
canonical N--adapted variables, {\small
\begin{eqnarray}
\widetilde{\mathcal{F}}(\widetilde{\mathbf{g}},\widetilde{\mathcal{D}},%
\breve{f}) &=&\int_{\overline{\mathcal{V}}}(~^{s}\widetilde{\mathbf{R}}+|h%
\widetilde{\mathcal{D}}\breve{f}|^{2}+|\ v\widetilde{\mathcal{D}}\breve{f}%
|)^{2})e^{-\breve{f}}\ dv,  \label{2npf1} \\
\ \widetilde{\mathcal{W}}(\widetilde{\mathbf{g}},\widetilde{\mathcal{D}},%
\breve{f},\breve{\tau}) &=&\int_{\overline{\mathcal{V}}}[\breve{\tau}(~^{s}%
\widetilde{\mathbf{R}}+|h\widetilde{\mathcal{D}}\breve{f}|+|v\widetilde{%
\mathcal{D}}\breve{f}|)^{2}+\breve{f}-2m]\breve{\mu}dv,  \label{2npf2}
\end{eqnarray}%
} where the new scaling function $\breve{f}$ is intorduced for $\int_{%
\overline{\mathcal{V}}}\breve{\mu}dv=1$ with volume element $dv,$ $\breve{\mu%
}=\left( 4\pi \breve{\tau}\right) ^{-m}e^{-\breve{f}}$ and $\breve{\tau}>0,$
where $\breve{\tau}=\ ^{h}\tau =\ ^{v}\tau $ \ for a couple of possible $h$%
-- and $v$--flows parameters, $\breve{\tau}=(\ ^{h}\tau ,\ ^{v}\tau ).$
\end{lemma}

\begin{proof}
On $\mathcal{T}^{E}\mathbf{P,}$ evolution models can be formulated as in
standard theory for Riemann metrics \cite{ham1,ham2,gper1,gper2,gper3} but
for a family of geometric data $\left( \overline{\mathbf{g}}(\tau ),%
\overline{\nabla }(\tau )\right) $ induced by a family of regular $\mathcal{L%
}(\tau )\in C^{\infty }(P)$ with a flow parameter $\tau \in \lbrack
-\epsilon ,\epsilon ]\subset \mathbb{R},$ when $\epsilon >0$ is taken
sufficiently small. For $P=E,$ we can postulate on the space of $Sec(E),$
for $\pi :E\rightarrow M,$ $\dim E=n+m\geq 3$ and $\dim M=n\geq 2,$ the
(Perelman's) functionals
\begin{eqnarray}
\ _{\shortmid }\mathcal{F}(\overline{\mathbf{g}},\overline{\nabla },f,\tau )
&=&\int_{\overline{\mathcal{V}}}\left( \ _{\shortmid }R+\left\vert \overline{%
\nabla }f\right\vert ^{2}\right) e^{-f}\ dv,  \label{2pfrs1} \\
\ _{\shortmid }\mathcal{W}(\overline{\mathbf{g}},\overline{\nabla },f,\tau )
&=&\int_{\overline{\mathcal{V}}}\left[ \tau \left( \ _{\shortmid
}R+\left\vert \overline{\nabla }f\right\vert \right) ^{2}+f-2m)\right] \mu \
dv,  \label{2pfrs2}
\end{eqnarray}%
where the volume form $dv$ and scalar curvature $\ _{\shortmid }R$ of $%
\overline{\nabla }$ are computed for sets off--diagonal metrics $g_{%
\overline{\alpha }\overline{\beta }}$ (\ref{offd}) with Euclidean singature.
The integration is taken over compact regions $\overline{\mathcal{V}}\subset
\mathcal{T}^{\mathbf{E}}\mathbf{E},\dim \mathcal{V}=2m,$ corresponding to
sections over a $U\subset M.$ We can fix $\int_{\overline{\mathcal{V}}}dv=1,$
with $\mu =\left( 4\pi \tau \right) ^{-m}e^{-f}$ and a real parameter $\tau
>0.$ We introduce a new function $\breve{f}$ instead of\ $\ f.$ The scalar
functions are re--defined in such a form that the \textquotedblright
sub--integral\textquotedblright\ formula (\ref{2pfrs1}) under the distortion
of the Ricci tensor (\ref{driccidist}) is re--written in terms of geometric
objects derived for the canonical d--connection, $(\ _{\shortmid
}R+\left\vert \overline{\nabla }f\right\vert ^{2})e^{-f}=(\ ^{s}\widetilde{%
\mathbf{R}}+|\widetilde{\mathcal{D}}\breve{f}|^{2})e^{-\breve{f}}\ +%
\widetilde{\Phi }$. We obtain the N--adapted functional (\ref{2npf1}). For
the second functional (\ref{2pfrs2}), we re--scale $\tau \rightarrow \breve{%
\tau}$ and write
\begin{equation*}
\left[ \tau (\ _{\shortmid }R+\left\vert \overline{\nabla }f\right\vert
)^{2}+f-2m\right] \mu =[\breve{\tau}(~\ ^{s}\widetilde{\mathbf{R}}+|h%
\widetilde{\mathcal{D}}\breve{f}|+|\ v\widetilde{\mathcal{D}}\breve{f}|)^{2}+%
\breve{f}-2m]\breve{\mu}+\widetilde{\Phi }_{1},
\end{equation*}%
for some $\widetilde{\Phi }$ and $\widetilde{\Phi }_{1}$ for which $\int_{%
\overline{\mathcal{V}}}\widetilde{\Phi }dv=0$ and $\int_{\overline{\mathcal{V%
}}}\widetilde{\Phi }_{1}dv=0.$ Finally, we get the formula (\ref{2npf2}). $%
\square $
\end{proof}

\vskip5pt

In the rest of this section, we shall only sketch the key points for proofs of theorems
when the geometric constructions are straightforward consequences of those
presented for the Levi--Civita connection in \cite%
{gper1,gper2,gper3,caozhu,kleiner,rbook} and extended to nonholonomic
configurations in \cite{vricci1,vricci2,vfinsler,vncricci,vfracrf}.
For our models, we consider operators which up to frame transforms are
defined by $\mathcal{L}$ via $\overline{\nabla }$ on $\mathcal{T}^{\mathbf{E}%
}\mathbf{E.}$ Following Remark \ref{remarkakdef} for distortions completely
defined by $\widetilde{\mathbf{g}},$ we can study canonical Cartan and
almost K\"{a}hler Ricci flows on prolongation Lie algebroids as nonholonomic
deformations of the Riemannian evolution.

Using $\tilde{\theta}(\cdot \mathbf{,}\cdot )\mathbf{:=\tilde{g}}(\mathcal{J}%
\cdot \mathbf{,}\cdot ),$ we can define the canonical (almost symplectic)
Laplacian operator, $\widetilde{\Delta }:=$ $\widetilde{\mathcal{D}}%
\widetilde{\mathcal{D}},$ and (the Levi--Civita) Laplace operator, $%
\overline{\Delta }=\overline{\nabla }\overline{\nabla },$ and consider a
parameter $\tau (\chi ),$ $\partial \tau /\partial \chi =-1.$ For
simplicity, we shall not include normalized terms or values of type $\int_{%
\overline{\mathcal{V}}}\widetilde{\Phi }_{1}dv=0$ if those values can be
generated, or transformed into gradient type ones, via nonholonomic
deformations. Inverting the distortion relations (\ref{akdralg}), we can
compute
\begin{eqnarray}
\overline{\Delta } &=&\widetilde{\Delta }+~^{Z}\widetilde{\Delta },\ \ ~^{Z}%
\widetilde{\Delta }=\widetilde{\mathcal{Z}}_{\overline{\alpha }}\widetilde{%
\mathcal{Z}}^{\overline{\alpha }}-[\widetilde{\mathcal{D}}_{\overline{\alpha
}}(\widetilde{\mathcal{Z}}^{\overline{\alpha }})+\widetilde{\mathcal{Z}}_{%
\overline{\alpha }}\widetilde{\mathcal{D}}^{\overline{\alpha }}];
\label{distb} \\
\ \overline{R}_{\ \overline{\beta }\overline{\gamma }} &=&\widetilde{\mathbf{%
R}}_{\ \overline{\beta }\overline{\gamma }}+\widetilde{\mathbf{Z}}ic_{%
\overline{\beta }\overline{\gamma }},\ \ _{\shortmid }R=~\ ^{s}\widetilde{%
\mathbf{R}}+\widetilde{\mathbf{g}}^{\overline{\beta }\overline{\gamma }}%
\widetilde{\mathbf{Z}}ic_{\overline{\beta }\overline{\gamma }}=~\ ^{s}%
\widetilde{\mathbf{R}}+\ ~\ ^{s}\widetilde{\mathbf{Z}},  \notag \\
\ ~\ ^{s}\widetilde{\mathbf{Z}} &=&\widetilde{\mathbf{g}}^{\overline{\beta }%
\overline{\gamma }}\ \widetilde{\mathbf{Z}}ic_{\overline{\beta }\overline{%
\gamma }}=\ h\widetilde{Z}+v\widetilde{Z},\ h\widetilde{Z}=\widetilde{%
\mathbf{g}}^{ab}\ \widetilde{\mathbf{Z}}ic_{ab},\ v\widetilde{Z}=\widetilde{%
\mathbf{g}}^{AB}\ \widetilde{\mathbf{Z}}ic_{AB};  \notag \\
~\ ^{s}\overline{R} &=&\ h\overline{R}+v\overline{R},\ \ h\overline{R}:=%
\widetilde{\mathbf{g}}^{ab}\ \overline{R}_{ab},\ v\overline{R}=\widetilde{%
\mathbf{g}}^{AB}\overline{R}_{AB},  \notag
\end{eqnarray}%
where the terms with left up label "Z" are determined by $\widetilde{%
\mathcal{Z}}$ (for instance, $\widetilde{\mathbf{Z}}ic_{\overline{\beta }%
\overline{\gamma }}$ are components of respective deformations of the Ricci
d--tensor). For convenience, the capital indices $A,B,C...$ are used for
distinguishing $v$--components even the prolongation Lie algebroid is
constructed for $\mathbf{P=E}.$

\subsection{N--adapted almost symplectic evolution equations}

Let us consider the symmetrization and anti--symmetrization operators, for
instance, $\mathbf{R}_{(\alpha \beta )\ }:=\frac{1}{2}(\mathbf{R}_{\alpha
\beta \ }+\mathbf{R}_{\beta \alpha \ })$ and $\mathbf{R}_{[\alpha \beta ]\
}:=\frac{1}{2}(\mathbf{R}_{\alpha \beta \ }-\mathbf{R}_{\beta \alpha \ }).$
Using deformations of type (\ref{distb}) of corresponding geometric values
in the proof, for instance, given in for Proposition 1.5.3 of \cite{caozhu},
we obtain

\begin{theorem}
\label{2theq1}a) The N--adapted Ricci flows for the Cartan d--connection $%
\widetilde{\mathcal{D}}$\ preserving a symmetric metric structure $\mathbf{%
\tilde{g}}$ on $\mathcal{T}^{E}\mathbf{E}$ can be characterized by this
system of geometric flow equations:
\begin{eqnarray}
\frac{\partial \widetilde{\mathbf{g}}_{ab}}{\partial \chi } &=&-(\widetilde{%
\mathbf{R}}_{ab\ }+\widetilde{\mathbf{Z}}ic_{ab}),\ \frac{\partial
\widetilde{\mathbf{g}}_{AB}}{\partial \chi }=-(\widetilde{\mathbf{R}}_{AB}+%
\widetilde{\mathbf{Z}}ic_{AB}),  \label{frham1} \\
\widetilde{\mathbf{R}}_{\ aA} &=&-\widetilde{\mathbf{Z}}ic_{aA},\ \widetilde{%
\mathbf{R}}_{\ Aa}=\ \widetilde{\mathbf{Z}}ic_{Aa},\   \label{frham1a} \\
\ \frac{\partial \widetilde{f}}{\partial \chi } &=&-(\widetilde{\Delta }%
+~^{Z}\widetilde{\Delta })\widetilde{f}+\left\vert \left( \widetilde{%
\mathcal{D}}-\widetilde{\mathcal{Z}}\right) \widetilde{f}\right\vert ^{2}-\
~\ ^{s}\widetilde{\mathbf{R}}-~\ ^{s}\widetilde{\mathbf{Z}},  \label{frham1b}
\end{eqnarray}%
and the property that {\small
\begin{eqnarray*}
&&\frac{\partial }{\partial \chi }\mathcal{F}(\widetilde{\mathbf{g}},%
\widetilde{\mathcal{D}},\widetilde{f})=\int_{\overline{\mathcal{V}}}[|%
\widetilde{\mathbf{R}}_{ab\ }+\widetilde{\mathbf{Z}}ic_{ab}+(\widetilde{%
\mathcal{D}}_{a}-\widetilde{\mathcal{Z}}_{a})(\widetilde{\mathcal{D}}_{b}-%
\widetilde{\mathcal{Z}}_{b})\widetilde{f}|^{2}+ \\
&&|\widetilde{\mathbf{R}}_{AB\ }+\widetilde{\mathbf{Z}}ic_{AB}+(\widetilde{%
\mathcal{D}}_{A}-\widetilde{\mathcal{Z}}_{A})(\widetilde{\mathcal{D}}_{B}-%
\widetilde{\mathcal{Z}}_{B})\widetilde{f}|^{2}]e^{-\widetilde{f}}dv,~\int_{%
\overline{\mathcal{V}}}e^{-\widetilde{f}}dv=const.
\end{eqnarray*}%
}b) In almost symplectic variables with $\tilde{\theta}=\ \tilde{g}%
_{ab}\delta ^{a}\wedge X^{b}$ and almost K\"{a}hler d--algebroids $\mathcal{K%
}^{\mathbf{E}}\mathbf{E},$ and for re--defined scaling function $\widetilde{f%
},$ up to normalizing terms, the $h$- and $v$--evolution equations are
written in equivalent form%
\begin{equation}
\frac{\partial \tilde{\theta}_{ab}}{\partial \chi }=-\widetilde{\mathbf{R}}%
_{[ab]},\ \frac{\partial \tilde{\theta}_{AB}}{\partial \chi }=-\widetilde{%
\mathbf{R}}_{[AB]}.  \label{frham2}
\end{equation}
\end{theorem}

For different classes of distortions of type (\ref{distb}), we can redefine
the scaling functions from above Lemma and write the evolution equations (%
\ref{frham1}) in the form (\ref{frham2}) for symplectic variables with (\ref%
{canasf}). On $\mathcal{T}^{E}\mathbf{E,}$ the corresponding system of Ricci
flow evolution equations can be written for $\widehat{\mathcal{D}},\mathbf{%
\mathbf{\ }}$
\begin{eqnarray}
\frac{\partial \widetilde{\mathbf{g}}_{ab}}{\partial \chi } &=&-2\widehat{%
\mathbf{R}}_{ab\ },\frac{\partial \widetilde{\mathbf{g}}_{AB}}{\partial \chi
}=-2\widehat{\mathbf{R}}_{AB},  \label{rfcandc} \\
\widehat{\mathbf{R}}_{\ aA} &=&0,\ \widehat{\mathbf{R}}_{\ Aa}=0,\ \ \frac{%
\partial \widehat{f}}{\partial \chi }=-\widehat{\Delta }\widehat{f}%
+\left\vert \widehat{\mathcal{D}}\widehat{f}\right\vert ^{2}-\ ^{s}\widehat{%
\mathbf{R}},  \notag
\end{eqnarray}%
which can be derived from the functional $\widehat{\mathcal{F}}(\widetilde{%
\mathbf{g}},\widehat{\mathcal{D}},\widehat{f})=~\int_{\overline{\mathcal{V}}%
}(~\ ^{s}\widehat{\mathbf{R}}+|\widehat{\mathcal{D}}\widehat{f}|^{2})$ $e^{-%
\widehat{f}}\ dv.$ We note that the conditions of type $\widehat{\mathbf{R}}%
_{\alpha A}=0$ and $\widehat{\mathbf{R}}_{A\alpha }=0$ must be imposed in
order to model N--adapted evolution scenarios only with symmetric metrics.
In general, a nonholonomically constrained evolution can result in
nonsymmetric metrics.

\begin{corollary}
The geometric almost K\"{a}hler d--algebroid evolution defined in Theorem %
\ref{2theq1} is characterized by corresponding flows (for all time $\tau \in
\lbrack 0,\tau _{0})$) of N--adapted frames, $\mathbf{\tilde{e}}_{\overline{%
\alpha }}(\tau )=\ \mathbf{\tilde{e}}_{\overline{\alpha }}^{\ \overline{%
\alpha }^{\prime }}(\tau ,x^{i},y^{C})\partial _{\overline{\alpha }^{\prime
}}$, which up to frame/ coordinate transforms are defined by the
coefficients
\begin{eqnarray*}
\ \mathbf{\tilde{e}}_{\overline{\alpha }}^{\ \overline{\alpha }^{\prime
}}(\tau ,x^{i},y^{C}) &=&\left[
\begin{array}{cc}
\ e_{a}^{\ a^{\prime }}(\tau ,x^{i},y^{C}) & ~\tilde{N}_{b}^{B}(\tau
,x^{i},y^{C})\ e_{B}^{\ a^{\prime }}(\tau ,x^{i},y^{C}) \\
0 & \ e_{A}^{\ A^{\prime }}(\tau ,x^{i},y^{C})%
\end{array}%
\right] ,\  \\
\mathbf{\tilde{e}}_{\ \overline{\alpha }^{\prime }}^{\overline{\alpha }%
}(\tau ,x^{i},y^{C})\ &=&\left[
\begin{array}{cc}
e_{\ a^{\prime }}^{a}=\delta _{a^{\prime }}^{a} & e_{\ \underline{i}}^{b}=-%
\tilde{N}_{b}^{B}(\tau ,x^{i},y^{C})\ \ \delta _{a^{\prime }}^{b} \\
e_{\ A^{\prime }}^{a}=0 & e_{\ A^{\prime }}^{A}=\delta _{A^{\prime }}^{A}%
\end{array}%
\right] ,
\end{eqnarray*}%
with $\tilde{g}_{ab}(\tau )=\ e_{a}^{\ a^{\prime }}(\tau ,x^{i},y^{C})\
e_{b}^{\ b^{\prime }}(\tau ,x^{i},y^{C})\eta _{a^{\prime }b^{\prime }}$ and
\newline
$\tilde{g}_{AB}(\tau )=\ e_{A}^{\ A^{\prime }}(\tau ,x^{i},y^{C})\ e_{B}^{\
B^{\prime }}(\tau ,x^{i},y^{C})\eta _{A^{\prime }B^{\prime }}$, where $\eta
_{a^{\prime }b^{\prime }}=diag[1,...,1]$ and $\eta _{A^{\prime }B^{\prime
}}=diag[1,...,1]$ in order to fix a Riemannian signature of $\ \mathbf{%
\tilde{g}}_{\alpha \beta }^{[0]}(x^{i},y^{C}),$ is given by equations $\frac{%
\partial }{\partial \tau }\mathbf{\tilde{e}}_{\ \overline{\alpha }^{\prime
}}^{\overline{\alpha }}\ =\ \mathbf{\tilde{g}}^{\overline{\alpha }\overline{%
\beta }}~\widetilde{\mathbf{R}}_{\overline{\beta }\overline{\gamma }}~\
\mathbf{\tilde{e}}_{\ \overline{\alpha }^{\prime }}^{\overline{\gamma }} $
if we prescribe that the geometric constructions are derived by the Cartan
d--connection.
\end{corollary}

The proof of this Corollary for $\mathcal{K}^{\mathbf{E}}\mathbf{E}$ is
similar to those presented in N--adapted forms for nonholonomic Ricci flows
and/or Finsler--Ricci evolution, or on $\mathcal{T}^{E}\mathbf{E,}$ see \cite%
{vricci1,vricci2,vfinsler,vncricci}. All constructions depend on the
type of d--connection we chose for our considerations.

\subsection{Functionals for entropy and thermodynamics $\mathcal{K}^{\mathbf{%
E}}\mathbf{E}$}

For three dimensional Ricci flows of Riemannian metrics, the value $\
_{\shortmid }\mathcal{W}$ (\ref{2pfrs2}) was introduced by G.\ Perelman \cite%
{gper1} as a "minus entropy" functional. We can consider that $\widetilde{%
\mathcal{W}}$ (\ref{2npf2}) has a similar interpretation but in almost
symplectic variables and on prolongation Lie d--algebroids. The main
equations stated by Theorem \ref{2theq1} for $\widetilde{\mathcal{F}}$ (\ref%
{2npf1}) \ can be proven in equivalent form.

\begin{theorem}
\label{2theveq}The Ricci flow evolution equations with symmetric metrics and
respective almost symplectic forms on $\mathcal{T}^{\mathbf{E}}\mathbf{E}$
and, correspondingly, $\mathcal{K}^{\mathbf{E}}\mathbf{E}$, see (\ref{frham1}%
), (\ref{frham1a}) and (\ref{frham2}), and functions $\widehat{f}(\chi )$
and $\widehat{\tau }(\chi )$ being solutions of
\begin{equation*}
\ \frac{\partial \tilde{f}}{\partial \chi }=-(\widetilde{\Delta }+\ ^{Z}%
\tilde{\Delta})\tilde{f}+\left\vert (\widetilde{\mathcal{D}}_{a}-\widetilde{%
\mathcal{Z}}_{a})\tilde{f}\right\vert ^{2}-\ ~\ ^{s}\widetilde{\mathbf{R}}\ +%
\frac{2m}{\hat{\tau}},\ \frac{\partial \tilde{\tau}}{\partial \chi }=-1,
\end{equation*}%
can be derived for a functional $\widetilde{\mathcal{W}}$ satisfying the
condition {\small
\begin{eqnarray*}
\frac{\partial }{\partial \chi }\mathcal{\tilde{W}}(\widetilde{\mathbf{g}}%
(\chi )\mathbf{,}\tilde{f}(\chi ),\tilde{\tau}(\chi ))&=&2\int_{\mathcal{V}}%
\tilde{\tau}[|\widetilde{\mathbf{R}}_{\overline{\alpha }\overline{\beta }\ }-%
\widetilde{\mathbf{Z}}ic_{\overline{\alpha }\overline{\beta }}\ + (%
\widetilde{\mathcal{D}}_{\overline{\alpha }}-\widetilde{\mathcal{Z}}_{%
\overline{\alpha }})(\widetilde{\mathcal{D}}_{\overline{\beta }}-\widetilde{%
\mathcal{Z}}_{\overline{\beta }})\tilde{f} \\
&&{\qquad} - \frac{1}{2\tilde{\tau}}\mathbf{\tilde{g}}_{\overline{\alpha }%
\overline{\beta }}|^{2}] (4\pi \tilde{\tau})^{-m}e^{-\tilde{f}}dv,
\end{eqnarray*}%
} for $\int_{\mathcal{V}}e^{-\tilde{f}}dv=const.$ Such a functional is
N--adapted and nondecreasing if it is both h-- and v--nondecreasing.
\end{theorem}

\begin{proof}
For the Levi--Civita connection on $\mathcal{T}^{\mathbf{E}}\mathbf{E,}$ the
proof is similar to that in Proposition 1.5.8 in \cite{caozhu} containing
the details of the original result from \cite{gper1}. Using N--adapted
deformations, the geometric constructions are performed in almost symplectic
variables on $\mathcal{K}^{\mathbf{E}}\mathbf{E.}$\ $\square $
\end{proof}

\vskip5pt

Let us remember some main concepts from statistical thermodynamics. It is
considered a partition function $Z=\int \exp (-\beta E)d\omega (E)$ for a
canonical ansamble at temperature $\beta ^{-1}.$ Such a temperature is
defined by the measure determined by the density of states $\omega (E).$ We
can provide a statistical analogy computing respective thermodynamical
values. In standard form, there are introduced $\ \left\langle
E\right\rangle :=-\partial \log Z/\partial \beta ,$ the entropy $S:=\beta
\left\langle E\right\rangle +\log Z$ and the fluctuation $\sigma
:=\left\langle \left( E-\left\langle E\right\rangle \right)
^{2}\right\rangle =\partial ^{2}\log Z/\partial \beta ^{2}.$ The original
idea of G. Perelman was to use such values for characterizing Ricci flows of
Riemannian metrics \cite{gper1}. The constructions can be elaborated in
N--adapted form for geometric flows subjected to non--integrable constraints
on various spaces endowed with nonholonomic distributions of commutative and
noncommutative type, Lie algebroids etc \cite%
{vricci1,vricci2,vfinsler,vncricci,vfracrf}.

\begin{theorem}
\label{theveq} The N--adapted metric compatible (with symmetric metrics)
Ricci on $\mathcal{T}^{\mathbf{E}}\mathbf{E}$ are characterized by a)\
canonical thermodynamic values
\begin{eqnarray}
\left\langle \widehat{E}\right\rangle \ &=&-\widehat{\tau }^{2}\int_{%
\mathcal{V}}(~\ ^{s}\widehat{\mathbf{R}}+|\widehat{\mathcal{D}}\widehat{f}%
|^{2}-\frac{m}{\widehat{\tau }})\widehat{\mu }\ dv,  \notag \\
\widehat{S} &=&-\int_{\mathcal{V}}[\widehat{\tau }(~\ ^{s}\widehat{\mathbf{R}%
}+|\widehat{\mathcal{D}}\widehat{f}|^{2})+\widehat{f}-2m]\widehat{\mu }\ dv,
\notag \\
\widehat{\sigma } &=&2\ \widehat{\tau }^{4}\int_{\mathcal{V}}[|\widehat{%
\mathbf{R}}_{\overline{\alpha }\overline{\beta }}-\widehat{\mathbf{Z}}ic_{%
\overline{\alpha }\overline{\beta }}+(\widetilde{\mathcal{D}}_{\overline{%
\alpha }}\mathbf{-}\ \widehat{\mathcal{Z}}_{\overline{\alpha }})(\widehat{%
\mathcal{D}}_{\overline{\beta }}\mathbf{-}\ \widehat{\mathcal{Z}}_{\overline{%
\beta }})\widehat{f}-\frac{1}{2\tilde{\tau}}\mathbf{g}_{\overline{\alpha }%
\overline{\beta }}|^{2}]\widehat{\mu }\ dv  \notag
\end{eqnarray}%
b) and/or by effective Lagrange and/or almost K\"{a}hler Ricci flows
\begin{eqnarray*}
\left\langle \tilde{E}\right\rangle &=&-\tilde{\tau}^{2}\int_{\overline{%
\mathcal{V}}}(\ ~\ ^{s}\widetilde{\mathbf{R}}+|\widetilde{\mathcal{D}}\tilde{%
f}|^{2}-\frac{m}{\tilde{\tau}})\tilde{\mu}\ dv, \\
\tilde{S} &=&-\int_{\overline{\mathcal{V}}}[\tilde{\tau}(~\ ^{s}\widetilde{%
\mathbf{R}}+|\widetilde{\mathcal{D}}\tilde{f}|^{2})+\tilde{f}-2m]\tilde{\mu}%
\ dv, \\
\tilde{\sigma} &=&2\ \tilde{\tau}^{4}~\int_{\overline{\mathcal{V}}}[|%
\widetilde{\mathbf{R}}_{\overline{\alpha }\overline{\beta }}+\widetilde{%
\mathcal{D}}_{\overline{\alpha }}\widetilde{\mathcal{D}}_{\overline{\beta }}%
\tilde{f}-\frac{1}{2\tilde{\tau}}\mathbf{\tilde{g}}_{\overline{\alpha }%
\overline{\beta }}|^{2}]\tilde{\mu}\ dv,
\end{eqnarray*}%
where all values are constructed equivalently in Cartan and/or almost
symplectic variables on $\mathcal{K}^{\mathbf{E}}\mathbf{E.}$\
\end{theorem}

\begin{proof}
Similar proofs in coordinate and/or N--adapted forms are given in \cite%
{caozhu,vricci1,vricci2,vfinsler,vncricci,vfracrf}. We have to use
the corresponding partition function $\tilde{Z}=\exp \left\{ ~\int_{%
\overline{\mathcal{V}}}[-\tilde{f}+m]~\tilde{\mu}dv\right\} $ for
computations on $\mathcal{K}^{\mathbf{E}}\mathbf{E.}$ The formulas in the
conditions of Theorem depend on the type of d--connection, $\nabla
\rightarrow $ $\widehat{\mathcal{D}}$, or $\nabla \rightarrow \widetilde{%
\mathcal{D}}$, which is chosen for nonholonomic deformations.\ Corresponding
re-scaling $\breve{f}\rightarrow \tilde{f},$ or $\widehat{f},$ and $\breve{%
\tau}\rightarrow \tilde{\tau},$ or $\widehat{\tau }$, have to be
considered.\ $\square $
\end{proof}

\vskip5pt

Finally, we note that Ricci flows with different d--connections are
characterized by different thermodynamical values and stationary
configurations.

\section{Ricci Solitons with Lie Algebroid Symmetries}

\label{s4}

In this section, we shall construct in explicit form some examples of exact
solutions for Ricci soliton Lie d--algebroid configurations. The fist class
of models describes generalized Einstein spaces with nonholonomic (for
instance, almost symplectic) variables and the second one is determined by
Lagrange--Finsler generating functions.

\subsection{Preliminaries on Lie d--algebroid solitons}

Lie d--algebroid Ricci solitons can be viewed as fixed points of generalized
Ricci flows with a functional $\widetilde{\mathcal{W}}$\ (\ref{2npf2})
satisfying the conditions of Theorem \ref{2theveq}. Such nonholonomically
constrained dynamical systems correspond to self--similar solutions
describing N--adapted geometric evolution models.

\begin{definition}
\label{defasstr}The geometric data $[\overline{\mathbf{g}}\sim \mathbf{%
\tilde{g},}\mathcal{L},\widetilde{\mathcal{N}},\widetilde{\mathcal{D}}%
]\approx \lbrack \tilde{\theta}(\cdot ,\cdot )\mathbf{:=\tilde{g}}(%
\widetilde{\mathcal{J}}\cdot \mathbf{,}\cdot ),$ $\ ^{\theta }\widetilde{%
\mathcal{D}}=\overline{\mathbf{\nabla }}+\widetilde{\mathcal{Z}}]$ for a
complete Riemannian metric $\overline{\mathbf{g}}$ on a smooth $\mathcal{T}^{%
\mathbf{E}}\mathbf{E}$ and corresponding $\mathcal{K}^{\mathbf{E}}\mathbf{E}$
define a gradient almost K\"{a}hler--Ricci d--algebroid soliton if there
exists a smooth potential function on $\tilde{\kappa}(x^{i},y^{C})$ such
that
\begin{equation}
\widetilde{\mathbf{R}}_{\ \overline{\beta }\overline{\gamma }}+\widetilde{%
\mathcal{D}}_{\overline{\beta }}\widetilde{\mathcal{D}}_{\overline{\gamma }}%
\tilde{\kappa}=\lambda \mathbf{\tilde{g}}_{\ \overline{\beta }\overline{%
\gamma }}.  \label{akrs}
\end{equation}%
Using the almost symplectic form (\ref{canasf}), these equations can be
written equivalently in the form
\begin{equation*}
\ ^{\theta }\widetilde{\mathbf{R}}_{\ [\overline{\beta }\overline{\gamma }%
]}+\ ^{\theta }\widetilde{\mathcal{D}}_{[\overline{\beta }}\ ^{\theta }%
\widetilde{\mathcal{D}}_{\overline{\gamma ]}}\tilde{\kappa}=\lambda \tilde{%
\theta}_{\ \overline{\beta }\overline{\gamma }}.
\end{equation*}%
There are three types of such Ricci solitons determined by $\lambda =const:$%
\ steady ones, for $\lambda =0;$ shrinking, for $\lambda >0;$ and expanding,
for $\lambda <0.$
\end{definition}

The above classification is important because shrinking solutions for the
Riemannian Levi--Civita solitons helps us to understand the asymptotic
behaviour of ancient solutions of Ricci flows (see, for instance,
Proposition 11.2 in \cite{gper1} and/or Theorem 6.2.1 in \cite{caozhu}). In
general, complete gradient shrinking Ricci solitions describe possible Type
I singularity models in the Ricci flow theory. If $\tilde{\kappa}=const,$
the equations (\ref{akrs}) transform into  distorted Einstein equations  but for Ricci solitonic configurations.

\begin{proposition}
Let $(\overline{\mathbf{g}}\sim \mathbf{\tilde{g},}\mathcal{L},\widetilde{%
\mathcal{N}},\widetilde{\mathcal{D}};\tilde{\kappa})$ be a complete
shrinking soliton on $\mathcal{T}^{\mathbf{E}}\mathbf{E}$ and/ or $\mathcal{K%
}^{\mathbf{E}}\mathbf{E.}$ Using nonholonomic frame deformations, we can
construct a redefined potential function $\widehat{\kappa }(x^{i},y^{C}),$
for $\overline{\mathbf{g}}\sim \mathbf{\tilde{g},}$ when (\ref{akrs}) are
equivalent to
\begin{equation}
\widehat{\mathbf{R}}_{\ \overline{\beta }\overline{\gamma }}+\widehat{%
\mathcal{D}}_{\overline{\beta }}\widehat{\mathcal{D}}_{\overline{\gamma }}%
\widehat{\kappa }=\lambda \overline{\mathbf{g}}_{\ \overline{\beta }%
\overline{\gamma }}.  \label{akrcan}
\end{equation}
\end{proposition}

\begin{proof}
Using Conclusion \ref{cequiv} and contracting indices in (\ref{akrs}), we
obtain that $~\ ^{s}\widetilde{\mathbf{R}}+|\widetilde{\mathcal{D}}\tilde{%
\kappa}|^{2}-$ $\tilde{\kappa}$\ $=const.$ Distortion relations of type $%
\widetilde{\mathcal{D}}_{\overline{\alpha }}=\ \widehat{\mathcal{D}}_{%
\overline{\alpha }}\mathbf{+}\ \widehat{\mathcal{Z}}_{\overline{\alpha }}$
allows us to compute $\ ^{s}\widehat{\mathbf{R}}+~\ ^{s}\widehat{\mathbf{Z}}%
+|(\ \widehat{\mathcal{D}}+\ \widehat{\mathcal{Z}})\tilde{\kappa}|^{2}-%
\tilde{\kappa}\ =const,$ which can be rewritten as $\ ^{s}\widehat{\mathbf{R}%
}+|\ \widehat{\mathcal{D}}\widehat{\kappa }|^{2}-\widehat{\kappa }=const$
for certain nonlinear transform $\tilde{\kappa}\rightarrow \widehat{\kappa }%
. $ In general, the systems (\ref{akrcan}) and (\ref{akrs}) have different
solutions. Nevertheless, conditions of type $\widehat{\mathcal{Z}}=0$ and/or
$\widetilde{\mathcal{Z}}=0$ result in the Levi--Civita configurations and
equivalent classes of solutions. $\square $
\end{proof}

\subsection{Generalized Einstein eqs encoding Lie d--algebroids}

We can construct very general classes of off--diagonal solutions of (\ref%
{akrcan}) if we impose the condition that in some N--adapted frames
\begin{eqnarray}
\widehat{\mathcal{D}}_{\overline{\gamma }}\widehat{\kappa } &=&\mathbf{e}_{%
\overline{\gamma }}\widehat{\kappa }=\kappa _{\overline{\gamma }}=const,
\label{constrpotent} \\
\mbox{ i.e. }\delta _{a}\widehat{\kappa } &=&\mathcal{X}_{a}\widehat{\kappa }%
-\mathcal{N}_{a}^{C}\kappa _{C}=0\mbox{ and }\mathcal{V}_{A}\widehat{\kappa }%
=\kappa _{A}.  \notag
\end{eqnarray}%
The information from potential functions $\widehat{\kappa }$ is encoded into
the data for N--connection structure with coefficients $\mathcal{N}_{a}^{C}.$

For simplicity, we shall consider in this section nonholonomic distributions
on a nonholonomic $\mathbf{E}=\mathbf{P}$ with $2+2$ splitting when $%
a,b,...=1,2;i^{\prime },j^{\prime },...=1,2$ and $A,B,...=3,4.$ The local
coordinates are parameterized in the form $u^{\mu
}=(x^{i},y^{a})=(x^{1},x^{2},y^{3},y^{4}).$ We study nonholonomic
deformations of a d--metric $\mathbf{\mathring{g}}$ on $\mathbf{E}$ into a
target metric $\overline{\mathbf{g}}$ (\ref{dm}) on $\mathcal{T}^{\mathbf{E}}%
\mathbf{E,}$ $\mathbf{\mathring{g}\rightarrow }\overline{\mathbf{g}},$ which
results in solutions of the Ricci solitonic equations (\ref{akrcan}) and (%
\ref{constrpotent}). The prime metric is parameterized
\begin{eqnarray}
\mathbf{\mathring{g}} &=&\mathring{g}_{\alpha }(u)\mathbf{\mathring{e}}%
^{\alpha }\otimes \mathbf{\mathring{e}}^{\beta }=\mathring{g}%
_{i}(x)dx^{i}\otimes dx^{i}+\mathring{h}_{a}(x,y)\mathbf{\mathring{e}}%
^{a}\otimes \mathbf{\mathring{e}}^{a},  \notag \\
\mbox{ for } &&\mathbf{\mathring{e}}^{\alpha }=(dx^{i},\mathbf{e}^{a}=dy^{a}+%
\mathring{N}_{i}^{a}(u)dx^{i}),  \label{pm} \\
&&\mathbf{\mathring{e}}_{\alpha }=(\mathbf{\mathring{e}}_{i}=\partial
/\partial y^{a}-\mathring{N}_{i}^{b}(u)\partial /\partial y^{b},\ {e}%
_{a}=\partial /\partial y^{a}).  \notag
\end{eqnarray}%
For physical applications, we can consider that the coefficients of such
metrics are with two Killing vector symmetries and that in certain systems
of coordinates it can be diagonalized\footnote{%
this include the bulk of physically important exact solutions of Einstein
equations}. In general, we can consider arbitrary (semi) Riemannian metrics.
The target Lie algebroid d--metrics are chosen
\begin{eqnarray}
\overline{\mathbf{g}} &=&\mathbf{g}_{\overline{\alpha }\overline{\beta }}%
\mathbf{e}^{\overline{\beta }}\otimes \mathbf{e}^{\overline{\beta }}=\
\mathbf{g}_{a}\ \mathcal{X}^{a}\otimes \mathcal{X}^{a}+\ \mathbf{g}_{A}\
\delta ^{A}\otimes \delta ^{A}  \label{tdm} \\
&=&\eta _{a}(x^{k})\mathring{g}_{a}\mathcal{X}^{a}\otimes \mathcal{X}%
^{a}+\eta _{A}(x^{k},y^{3})\mathring{h}_{A}\delta ^{A}\otimes \delta ^{A},
\notag
\end{eqnarray}%
where we shall construct exact solutions with Killing symmetry on $\partial
/\partial y^{4}$ (non--Killing configurations request a more advanced
geometric techniques). Let us denote by $h^{\ast }:=\partial _{3}h$ and $%
\mathcal{N}_{a}^{3}=w_{a}(x^{k},y^{3})$, $\mathcal{N}%
_{a}^{4}=n_{a}(x^{k},y^{3}).$

\begin{proposition}
The nontrivial components of the Ricci soliton d--algebroid equations (\ref%
{akrcan}) and (\ref{constrpotent}), with respect to N--adapted bases (\ref%
{dderalg}), (\ref{ddifalg}) and for coordinate transforms when $\partial
_{a}\rightarrow \mathcal{X}_{a}$ and $\mathcal{V}_{A}=\partial _{A}$ for a
metric (\ref{tdm}, are {\small
\begin{eqnarray}
-\widehat{\mathbf{R}}_{1}^{1} &=&-\widehat{\mathbf{R}}_{2}^{2}=\frac{1}{%
2g_{1}g_{2}}[\mathcal{X}_{1}(\mathcal{X}_{1}g_{2})-\frac{\mathcal{X}%
_{1}g_{1}\ \mathcal{X}_{1}g_{2}}{2g_{1}}-\frac{\left( \mathcal{X}%
_{1}g_{2}\right) ^{2}}{2g_{2}}  \label{eq1b} \\
&& +\mathcal{X}_{2}(\mathcal{X}_{2}g_{1})-\frac{\mathcal{X}_{2}g_{1}\
\mathcal{X}_{2}g_{2}}{2g_{2}}-\frac{(\mathcal{X}_{2}g_{1})^{2}}{2g_{1}}]
=\lambda ,  \notag \\
-\widehat{\mathbf{R}}_{3}^{3} &=&-\widehat{\mathbf{R}}_{4}^{4}=\frac{1}{%
2h_{3}h_{4}}[h_{4}^{\ast \ast }-\frac{\left( h_{4}^{\ast }\right) ^{2}}{%
2h_{4}}-\frac{h_{3}^{\ast }h_{4}^{\ast }}{2h_{3}}]=\lambda ,  \label{eq2b} \\
\widehat{\mathbf{R}}_{3a} &=&\frac{w_{a}}{2h_{4}}[h_{4}^{\ast \ast }-\frac{%
\left( h_{4}^{\ast }\right) ^{2}}{2h_{4}}-\frac{h_{3}^{\ast }h_{4}^{\ast }}{%
2h_{3}}]+\frac{h_{4}^{\ast }}{4h_{4}}(\frac{\mathcal{X}_{a}h_{3}}{h_{3}}+%
\frac{\mathcal{X}_{a}h_{4}}{h_{4}})-\frac{\mathcal{X}_{a}h_{4}^{\ast }}{%
2h_{4}}=0,  \label{eq3b} \\
\widehat{\mathbf{R}}_{4a} &=&\frac{h_{4}}{2h_{3}}n_{a}^{\ast \ast }+(\frac{%
h_{4}}{h_{3}}h_{3}^{\ast }-\frac{3}{2}h_{4}^{\ast })\frac{n_{a}^{\ast }}{%
2h_{3}}=0;  \label{eq4b}
\end{eqnarray}%
}for the equations for the potential function%
\begin{equation*}
\mathcal{X}_{a}\widehat{\kappa }-w_{a}\kappa _{3}-n_{a}\kappa _{4}=0%
\mbox{
and }\mathcal{V}_{A}\widehat{\kappa }=\kappa _{A},
\end{equation*}%
when the torsionless (Levi--Civita, LC) conditions $\widehat{\mathcal{Z}}=0$
transform into
\begin{eqnarray}
w_{a}^{\ast } &=&(\mathcal{X}_{a}-w_{a}\partial _{3})\ln \sqrt{|h_{3}|},(%
\mathcal{X}_{a}-w_{a}\partial _{3})\ln \sqrt{|h_{4}|}=0,  \label{lccondb} \\
\mathcal{X}_{b}w_{a} &=&\mathcal{X}_{a}w_{b},n_{a}^{\ast }=0,\partial
_{a}n_{b}=\partial _{b}n_{a}.  \notag
\end{eqnarray}
\end{proposition}

\begin{proof}
It follows from straightforward computations of the Ricci d--tensor on $%
\mathcal{T}^{\mathbf{E}}\mathbf{E}$ using formulas (\ref{candcon}) ang (\ref%
{dcurv}). The local frames are re--defined in the form
$\partial _{a}=e_{\ a}^{a^{\prime }}\mathcal{X}_{a^{\prime }}$ in order to
include Lie algebroid anchor structure functions and commutation relations
of type (\ref{extercalc}) and (\ref{extercalc1}). For nonholonomic
2+2+2+...+ decompositions, this can be performed by local frame/coordinate
transforms. Details for calculus on $T\mathbf{V}$ are provided in \cite%
{vrflg} and references therein. We can work similarly both on $T%
\mathbf{V}$ and $\mathcal{T}^{\mathbf{E}}\mathbf{E}$ but with different
N--adapted nonholonomic frames and N--elongated partial derivatives and
differentials. The system of equations (\ref{eq1b}) - (\ref{lccondb})
possess an important decoupling property. For instance, the equation (\ref%
{eq1b}) is a 2--d version of Laplace/ d'Alambert equation (it depends on
signature of the $h$--metric) with prescribed local source $\lambda .$ Such
equations can be integrated in general form even the algebroid structures
functions $\rho _{a}^{i}(x^{k})$ are not trivial. The equation (\ref{eq2b})
is the same both on $T\mathbf{V}$ and $\mathcal{T}^{\mathbf{E}}\mathbf{E}$
and contains partial derivatives only on $\partial _{3}$ and can be also in
similar form. $\square $
\end{proof}

\subsection{Generating off--diagonal solutions}

We can integrate the algebroid Ricci soliton equations (\ref{akrcan}) and (%
\ref{constrpotent}) for a nontrivial source $\lambda $, $g_{a}=\epsilon
_{a}e^{\psi {(x^{k})}},$ $\epsilon _{a}=\pm 1$ and $h_{a}^{\ast }\neq 0.$

\begin{theorem}
The system (\ref{eq1b})--(\ref{eq4b}) decouple in N--adapted form,
\begin{eqnarray}
\epsilon _{1}\mathcal{X}_{1}(\mathcal{X}_{1}\psi )+\epsilon _{2}\mathcal{X}%
_{2}(\mathcal{X}_{2}\psi ) &=&2~\lambda  \label{eq1} \\
\phi ^{\ast }h_{4}^{\ast } &=&2h_{3}h_{4}\lambda  \label{eq2} \\
\beta w_{a}-\alpha _{a} &=&0,  \label{eq3} \\
n_{a}^{\ast \ast }+\gamma n_{a}^{\ast } &=&0,  \label{eq4}
\end{eqnarray}%
\begin{equation}
\mbox{ for } \alpha _{a}=h_{4}^{\ast }\mathcal{X} _{a}\phi ,\beta
=h_{4}^{\ast }\ \phi ^{\ast },\gamma =\left( \ln
|h_{4}|^{3/2}/|h_{3}|\right) ^{\ast },  \label{abc}
\end{equation}%
\begin{equation}
\mbox{ where } {\phi =\ln |h_{4}^{\ast }/\sqrt{|h_{3}h_{4}|}|}  \label{genf}
\end{equation}%
is considered as a generating function.
\end{theorem}

\begin{proof}
It follows from explicit computations for d--metrics (\ref{tdm}) with
Killing symmetry on $\partial _{4}.$ It is convenient to use the value $\Phi
:=e^{{\phi }}.$ $\square $
\end{proof}

\begin{corollary}
The above systems of nonlinear partial differential equations, PDE, can be
integrated in very general forms.
\end{corollary}

\begin{proof}
We should follow such a procedure:

\begin{enumerate}
\item The (\ref{eq1}) is just a 2--d Laplace/ d'Alambert equation which can
be solved for any given $\lambda .$

\item For $h_{A}:=\epsilon _{A}z_{A}^{2}(x^{k},y^{3}),$ $\epsilon _{A}=\pm 1$
(we do not consider summation on repeating indices in this formula), the
system of two equations (\ref{eq2}) and (\ref{genf}) can be written  $\phi
^{\ast }z_{4}^{\ast }=\epsilon _{3}z_{4}(z_{3})^{2}\lambda$ and  $e^{{\phi }%
}z_{3}=2\epsilon _{4}z_{4}^{\ast }$. Multiplying both equations for nonzero $%
z_{4}^{\ast },\phi ^{\ast },z_{A}$ and introducing the result instead of the
first equation, this system transforms into  $\Phi ^{\ast }=2\epsilon
_{3}\epsilon _{4}z_{3}z_{4}\lambda$ and  $\Phi z_{3}=2\epsilon
_{4}z_{4}^{\ast }$. Introducing $z_{3}$ from the second equation into the
first one, we obtain $[(z_{4})^{2}]^{\ast }=\epsilon _{3}[\Phi ^{2}]^{\ast
}/4\lambda .$ We can integrate on $y^{3},$
\begin{equation}
h_{4}=\epsilon _{4}(z_{4})^{2}=\ ^{0}h_{4}(x^{k})+\frac{\epsilon
_{3}\epsilon _{4}}{4\lambda }\Phi ^{2},  \label{h4a}
\end{equation}%
for an integration function $\ ^{0}h_{4}(x^{k}).$ From the first equation in
above system, we compute%
\begin{equation}
h_{3}=\epsilon _{3}(z_{3})^{2}=\frac{\phi ^{\ast }}{\lambda }\frac{%
z_{4}^{\ast }z_{4}}{z_{4}z_{4}}=\frac{1}{2\lambda }(\ln |\Phi |)^{\ast }(\ln
|h_{4}|)^{\ast }.  \label{h3a}
\end{equation}
Redefining the coordinates and $\Phi $ and introducing $\epsilon
_{3}\epsilon _{4}$ in $\lambda $ we express the solutions in functional
form, $h_{3}[\Phi ]=(\Phi ^{\ast })^{2}/\lambda \Phi ^{2},\ h_{4}[\Phi
]=\Phi ^{2}/4\lambda $. %\label{h4bl}
%\end{equation}

\item To find $w_{a}$ we have to solve certain algebraic equations which can
be obtained if we introduce the coefficients (\ref{abc}) in (\ref{eq3}),
\begin{equation}
w_{a}=\mathcal{X}_{a}\phi /\phi ^{\ast }=\mathcal{X}_{a}\Phi /\Phi ^{\ast }.
\label{w1bl}
\end{equation}

\item Integrating two times on $y^{3}$ in (\ref{eq4}), we find%
\begin{equation}
n_{b}=\ _{1}n_{b}+\ _{2}n_{b}\int dy^{3}\ h_{3}/(\sqrt{|h_{4}|})^{3},
\label{n1b}
\end{equation}%
where $\ _{1}n_{b}(x^{i}),\ _{2}n_{b}(x^{i})$ are integration functions.

\item The nonholonomic constraints for the LC--conditions (\ref{lccondb})
can be solved in explicit form for certain classes of integration functions $%
\ _{1}n_{b}$ and $\ _{2}n_{b}.$ We can find explicit solutions if $\
_{2}n_{b}=0$ and $\ _{1}n_{b}=\mathcal{X}_{b}n$ with a function $n=n(x^{k}).$
We get $(\mathcal{X}_{a}-w_{a}\partial _{3})\Phi \equiv 0$ for any $\Phi
(x^{k},y^{3})$ if $w_{a}$ is defined by (\ref{w1bl}). For any functional $%
H(\Phi ),$ we obtain $(\mathcal{X}_{a}-w_{a}\partial _{3})H=\frac{\partial H%
}{\partial \Phi }(\mathcal{X}_{a}-w_{a}\partial _{3})\Phi =0$. It is
possible to solve the equations $(\mathcal{X}_{a}-w_{a}\partial _{3})h_{4}=0$
for $h_{4}=H(|\tilde{\Phi}(\Phi )|)$. This way we solve the second system of
equations in (\ref{lccondb}) when $(\mathcal{X}_{a}-w_{a}\partial _{3})\ln
\sqrt{|h_{4}|}\sim (\mathcal{X}_{a}-w_{a}\partial _{3})h_{4}.$ We can
consider a subclass of generating functions $\Phi =\check{\Phi}$ for which  $%
(\mathcal{X}_{a}\check{\Phi})^{\ast }=\mathcal{X}_{a}(\check{\Phi}^{\ast })$%
. Then, we can compute for the left part of the second equation in (\ref%
{lccondb}), $(\mathcal{X}_{a}-w_{a}\partial _{3})\ln \sqrt{|h_{4}|}=0.$ The
first system of equations in (\ref{lccondb}) can be solved in explicit for
any $w_{a}$ determined by formulas (\ref{w1bl}), and $h_{3}[\tilde{\Phi}]$
and $h_{4}[\tilde{\Phi},\tilde{\Phi}^{\ast }]$. Let us consider $\tilde{\Phi}%
=\tilde{\Phi}(\ln \sqrt{|h_{3}|})$ for a functional dependence $h_{3}[\tilde{%
\Phi}[\check{\Phi}]].$ This allows us to obtain the formulas $w_{a}=\mathcal{%
X}_{a}|\tilde{\Phi}|/|\tilde{\Phi}|^{\ast }=\mathcal{X}_{a}|\ln \sqrt{|h_{3}|%
}|/|\ln \sqrt{|h_{3}|}|^{\ast }$. Taking derivative $\partial _{3}$ on both
sides of this equation, we get $w_{a}^{\ast }=\frac{(\mathcal{X}_{a}|\ln
\sqrt{|h_{3}|}|)^{\ast }}{|\ln \sqrt{|h_{3}|}|^{\ast }}-w_{a}\frac{|\ln
\sqrt{|h_{3}|}|^{\ast \ast }}{|\ln \sqrt{|h_{3}|}|^{\ast }}$. The condition $%
w_{a}^{\ast }=(\mathcal{X}_{a}-w_{a}\partial _{3})\ln \sqrt{|h_{3}|}$ is
necessary for the zero torsion conditions. It is satisfied for $\Phi =\check{%
\Phi}$. We can chose  $w_{a}=\check{w}_{a}=\mathcal{X}_{a}\check{\Phi}/%
\check{\Phi}^{\ast }=\mathcal{X}_{a}\widetilde{A}$, with a nontrivial
function $\widetilde{A}(x^{k},y^{3})$ depending functionally on generating
function $\check{\Phi}$ in order to solve the equations $\mathcal{X}%
_{a}w_{b}=\mathcal{X}_{b}w_{a}$ from the second line in (\ref{lccondb}).
\end{enumerate}

$\square $
\end{proof}

\begin{conclusion}
The class of off--diagonal metrics of type (\ref{tdm}) with coefficients
computed following the method outlined in above Proof are determined by
quadratic elements of type $ds^{2}=$ {\small
\begin{equation}
e^{\psi (x^{k})}[\epsilon _{1}(\mathcal{X}^{1})^{2}+\epsilon _{2}(\mathcal{X}%
^{2})^{2}]+ \epsilon _{3}\frac{(\check{\Phi}^{\ast })^{2}}{\lambda \check{%
\Phi}^{2}}[\mathcal{V}^{3}+(\mathcal{X}_{a}\widetilde{A}[\check{\Phi}])%
\mathcal{X}^{a}]^{2}+\epsilon _{4}\frac{\check{\Phi}^{2}}{4|\lambda |}[%
\mathcal{V}^{4}+(\mathcal{X}_{a}n)\mathcal{X}^{a}]^{2}.  \label{qelgen}
\end{equation}%
}
\end{conclusion}

In general, on prolongation Lie d--algebroids, the solutions defining Ricci
solitons can be with nontrivial torsion.

\begin{remark}
For arbitrary $\phi $ and related $\Phi ,$ or $\widetilde{\Phi },$ we can
generate off--diagonal solutions of (\ref{eq1b})--(\ref{eq4b}) with
nonholonomically induced torsion, {\small
\begin{eqnarray}
ds^{2} &=&e^{\psi (x^{k})}[\epsilon _{1}(\mathcal{X}^{1})^{2}+\epsilon _{2}(%
\mathcal{X}^{2})^{2}]+  \label{qelgent} \\
&&\epsilon _{3}(z_{3})^{2}[\mathcal{V}^{3}+\frac{\mathcal{X}_{a}\Phi }{\Phi
^{\ast }}\mathcal{X}^{a}]^{2}+\epsilon _{4}(z_{4})^{2}[\mathcal{V}^{4}+(\
_{1}n_{a}+\ _{2}n_{a}\int dy^{3}\frac{(z_{3})^{2}}{(z_{4})^{3}})\mathcal{X}%
^{a}]^{2},  \notag
\end{eqnarray}%
} where the values $z_{3}(x^{k},y^{3})$ and $z_{4}(x^{k},y^{3})$ are defined
by formulas (\ref{h3a}) and (\ref{h4a}). In N--adapted frames, the ansatz
for such solutions define a nontrivial distorting tensor $\widehat{%
\mathbf{Z}}=\{\widehat{\mathbf{Z}}_{\ \beta \gamma }^{\alpha }\}$,  see  (\ref{distrel1}).
\end{remark}

Taking data $\mathring{g}_{a}(x^{k})$ and $\mathring{h}_{A}(x^{k})$ for a
prime metric (\ref{pm}) to define, for instance, a black hole solution for
Einstein-- de Sitter spaces, and re--parameter\-iz\-ing the metric (\ref%
{qelgen}) in the form (\ref{tdm}), we can study nonholonomic deformations of
black hole metrics into Lie algebroid solitionic configurations. In explicit
form, such examples of algebroid black holes are proved in \cite{vhamlagralg}
and reference therein.

\subsection{On Lie algebroid \& almost K\"{a}hler -- Finsler Ricci solitons}

\label{ssfs}

We show how a Finsler geometry model can be nonholonomically deformed into a
Ricci soliton d--algebroid configuration. Let $\mathbf{E=TM}$ for a tangent
bundle $TM$ on the base space $M$ being a real $\mathit{C}^{\infty }$
manifold of dimension $\dim M=n=2.$

\begin{definition}
A Finsler fundamental, or generating, function (metric) is a function $F:\
TM\rightarrow \lbrack 0,\infty )$ for which\ 1) $F(x,y)$ is $\mathit{C}%
^{\infty }$ on $\widetilde{TM}:=TM\backslash \{0\},$ where $\{0\}$ is the
set of zero sections of $TM$ on $M;$\ 2) $F(x,\beta y)=\beta F(x,y),$ for
any $\beta >0,$ i.e. it is a positive 1--homogeneous function on the fibers
of $TM;$\ 3) for any $y\in \widetilde{T_{x}M},$ the Hessian $\ ^{v}{\tilde{g}%
}_{ij}(x,y)=\frac{1}{2}\frac{\partial ^{2}F^{2}}{\partial y^{i}\partial y^{j}%
} $ is considered as s a \textquotedblright vertical\textquotedblright\ (v)
metric on typical fiber, i.e. it is nondegenerate and positive definite, $%
\det |\ ^{v}{\tilde{g}}_{ij}|\neq 0.$
\end{definition}

For the conditions of Theorem \ref{theoranalog}, we take $\mathcal{L}%
=L=F^{2} $ and construct the geometric data $(\mathbf{\tilde{g},\tilde{N}}).$

\begin{theorem}
Any Finsler--Cartan geometry (equivalently modelled as an almost K\"{a}%
hler--Finsler space) with nonholonomic splitting $2+2$ can be encoded as a
canonical Ricci d--algebroid soliton with metric of type (\ref{qelgent}) and
a respective almost K\"{a}hler d--algebroid soliton, see Definition \ref%
{defasstr}.
\end{theorem}

\begin{proof}
It follows from explicit computations for d--metrics (\ref{tdm}) with
Killing symmetry on $\partial _{4}.$ We consider for the class of metrics (%
\ref{tdm}) that up to frame/ coordinate transforms the prime configuration
is determined by a total bundle metric $\mathbf{\mathring{g}}_{\overline{%
\alpha }^{\prime }\overline{\beta }^{\prime }}=e_{\ \overline{\alpha }%
^{\prime }}^{\overline{\alpha }}e_{\ \overline{\beta }^{\prime }}^{\overline{%
\beta }}\mathbf{\tilde{g}}_{\overline{\alpha }\overline{\beta }}.$ The
target metric $\overline{\mathbf{g}}_{\overline{\alpha }^{\prime }\overline{%
\beta }^{\prime }}$ (\ref{tdm}) defines generic off--diagonal solutions for
prolongation Lie d--algebroids with canonical d--connections if $\overline{%
\mathbf{g}}_{\overline{\alpha }^{\prime }\overline{\beta }^{\prime }}$ is of
type (\ref{qelgent}). Such solutions of Ricci soliton d--algebroid
equaitions (\ref{akrcan}) define nonholonomic transforms $(\mathbf{\mathring{%
g}}\sim \mathbf{\tilde{g},}\mathcal{L}=F^{2},\widetilde{\mathcal{N}},%
\widetilde{\mathcal{D}};\tilde{\kappa})\rightarrow (\overline{\mathbf{g}}%
\sim \widehat{\mathbf{g}}\mathbf{,}\widehat{\mathcal{N}},\widehat{\mathcal{D}%
};\widehat{\kappa })$. We re--encode a Finsler--Cartan geometry into the
canonical data for a Ricci soliton solution on $\mathcal{T}^{\mathbf{E}}%
\mathbf{E,}$ for $\mathbf{E=TM.}$ Via additional d--connection distortions $%
\ ^{\theta }\widetilde{\mathcal{D}}=\widehat{\mathcal{D}}+\mathcal{Z},$
completely defined by a Ricci soliton d--algebroid solution (\ref{qelgent}),
we re--define the geometric constructions on $\mathcal{K}^{\mathbf{TM}}%
\mathbf{TM.}$ For fundamental geometric values, $(\overline{\mathbf{g}}\sim
\widehat{\mathbf{g}},\widehat{\mathcal{N}},\widehat{\mathcal{D}};\widehat{%
\kappa })\approx [\tilde{\theta}(\cdot ,\cdot )\mathbf{:=\tilde{g}}(%
\widetilde{\mathcal{J}}\cdot \mathbf{,}\cdot ),\ ^{\theta }\widetilde{%
\mathcal{D}}=\widehat{\mathcal{D}}+\mathcal{Z}]$.

The constructions for this theorem can be extended for nonholonomic
splitting of any finite dimension 2+2...+2. $\square $
\end{proof}

\vskip5pt

\textbf{Acknowledgments:\ } The work is partially supported by the Program
IDEI, PN-II-ID-PCE-2011-3-0256. It contains the main results presented as a
Plenary Lecture at the 11th Panhellenic Geometry Conference (May 31 - June
2, 2013; Department of Mathematics of the National and Kapodistrian
University of Athens, Greece) and at the Joint International Meeting AMS -
Romanian MS (June 27-30, 22013; Alba Iulia, Romania). The author is grateful
to the Organizing Committees of Conferences for support. He thanks
Professors  S. Basilakos, K. Mackenzie, N. Mavromatos, E. N. Saridakis, P.
Stavrinos, F. Radulescu and D. Tataru for acceptance of talks and/or
important discussions.

\appendix

\setcounter{equation}{0} \renewcommand{\theequation}
{A.\arabic{equation}} \setcounter{subsection}{0}
\renewcommand{\thesubsection}
{A.\arabic{subsection}}

\section{Formulas in Coefficient Forms}

In this section, we summarize some important local constructions and
coefficient formulas which are necessary for formulating Ricci evolution
equations and deriving exact solutions on Lie algebroid models.

\subsection{Torsions and curvatures on $\mathcal{T}^{\mathbf{E}}\mathbf{P}$}

The N--adapted components $\mathbf{\Gamma }_{\ \overline{\beta }\overline{%
\gamma }}^{\overline{\alpha }}=(\mathbf{L}_{bf}^{a},\mathbf{L}_{Bf\;}^{A};%
\mathbf{B}_{bC}^{a},\mathbf{B}_{BC}^{A})$ of a\ d--connec\-tion \ref{defdcon}
and corresponding covariant operator $\mathcal{D}_{\overline{\alpha }}=(~%
\mathbf{e}_{\overline{\alpha }}\rfloor \mathcal{D}),$ where $\rfloor $ is
the interior product, are computed following equations $\mathbf{\Gamma }_{\
\overline{\alpha }\overline{\beta }}^{\overline{\gamma }}=(\mathcal{D}_{%
\overline{\alpha }}\mathbf{e}_{\overline{\beta }})\rfloor \mathbf{e}^{%
\overline{\gamma }}.$ There are defined the h-- and v--covariant
derivatives, respectively, $h\mathcal{D}=\{\mathcal{D}_{\gamma }=(\mathbf{L}%
_{bf}^{a},\mathbf{L}_{Bf\;}^{A})\}$ and $\ v\mathcal{D}=\{\mathcal{D}_{C}=(%
\mathbf{B}_{bC}^{a},\mathbf{B}_{BC}^{A})\},$ where {\small
\begin{equation*}
\mathbf{L}_{bf}^{a}:=(\mathcal{D}_{f}\delta _{b})\rfloor \mathcal{X}^{a},%
\mathbf{L}_{Bf}^{A}:=(\mathcal{D}_{f}\mathcal{V}_{B})\rfloor \delta ^{A},%
\mathbf{B}_{bC}^{a}:=(\mathcal{D}_{C}\delta _{b})\rfloor \mathcal{X}^{a},%
\mathbf{B}_{BC}^{A}:=(\mathcal{D}_{C}\mathcal{V}_{B})\rfloor \delta ^{A}
\end{equation*}%
} are computed for N--adapted bases (\ref{dderalg}) and (\ref{ddifalg}).

Using rules of absolute differentiation (\ref{extercalc}) for N--adapted
bases $~\mathbf{e}_{\overline{\alpha }}:=\{\delta _{\alpha },\mathcal{V}%
_{A}\}$ and $\mathbf{e}^{\overline{\beta }}:=\{\mathcal{X}^{\alpha },\delta
^{B}\}$ and the d--connection 1--form $\mathbf{\Gamma }_{\ \overline{\alpha }%
}^{\overline{\gamma }}:=\mathbf{\Gamma }_{\ \overline{\alpha }\overline{%
\beta }}^{\overline{\gamma }}\mathbf{e}^{\overline{\beta }},$ we can compute
the torsion and curvature 2--forms of $\mathcal{D}$ on $\mathcal{T}^{E}%
\mathbf{P}.$ For instance, let us consider such a calculus for the
d--torsion 2--form. We take some sections $\overline{x},\overline{y},%
\overline{z}$ \ of $\mathcal{T}^{E}\mathbf{P}$ parameterized in the form,
for example, $\overline{z}=z^{\overline{\alpha }}\mathbf{e}_{\overline{%
\alpha }}=z^{a}\delta _{a}+z^{A}\mathcal{V}_{A}.$ Following a N--adapted
differential form calculus, we prove some important formulas for the
d--torsion and d--curvature.

\begin{theorem}
For a d--connection $\mathcal{D},$ we can compute

\begin{itemize}
\item a)\ the torsion $\mathcal{T}^{\overline{\alpha }}:=\mathcal{D}\mathbf{e%
}^{\overline{\alpha }}=d\mathbf{e}^{\overline{\alpha }}+\mathbf{\Gamma }_{\
\overline{\beta }}^{\overline{\alpha }}\wedge \mathbf{e}^{\overline{\beta }}$%
; the $h$--$v$--coefficients $\mathcal{T}^{\overline{\alpha }}=\{\mathbf{T}%
_{\ \overline{\beta }\overline{\gamma }}^{\overline{\alpha }}\}=\{\mathbf{T}%
_{\ bf}^{a},\mathbf{T}_{\ bA}^{a},\mathbf{T}_{\ bf}^{A},\mathbf{T}_{\
Ba}^{A},\mathbf{T}_{\ BC}^{A}\}$ with N--adapted coefficients
\begin{eqnarray}
\mathbf{T}_{\ bf}^{a} &=&\mathbf{L}_{\ bf}^{a}-\mathbf{L}_{\ fb}^{a}+C_{\
bf}^{a},\ \mathbf{T}_{\ bA}^{a}=-\mathbf{T}_{\ Ab}^{a}=\mathbf{B}_{\
bA}^{a},\ \mathbf{T}_{\ ba}^{A}=\Omega _{\ ba}^{A},\   \notag \\
\mathbf{T}_{\ Ba}^{A} &=&\frac{\partial \mathcal{N}_{a}^{A}}{\partial u^{B}}-%
\mathbf{L}_{\ Ba}^{A},\ \mathbf{T}_{\ BC}^{A}=\mathbf{B}_{\ BC}^{A}-\mathbf{B%
}_{\ CB}^{A}.  \label{dtors}
\end{eqnarray}

\item b)\ The curvature $\mathcal{R}_{~\overline{\beta }}^{\overline{\alpha }%
}:=\mathcal{D}\mathbf{\Gamma }_{\ \overline{\beta }}^{\overline{\alpha }}=d%
\mathbf{\Gamma }_{\ \overline{\beta }}^{\overline{\alpha }}-\mathbf{\Gamma }%
_{\ \overline{\beta }}^{\overline{\gamma }}\wedge \mathbf{\Gamma }_{\
\overline{\gamma }}^{\overline{\alpha }}=\mathbf{R}_{\ \overline{\beta }%
\overline{\gamma }\overline{\delta }}^{\overline{\alpha }}\mathbf{e}^{%
\overline{\gamma }}\wedge \mathbf{e}^{\overline{\delta }}$, where $\mathbf{R}%
_{\ ~\overline{\beta }\overline{\gamma }\overline{\delta }}^{\overline{%
\alpha }}=\mathbf{e}_{\overline{\delta }}\mathbf{\Gamma }_{\ \overline{\beta
}\overline{\gamma }}^{\overline{\alpha }}-\mathbf{e}_{\overline{\gamma }}\
\mathbf{\Gamma }_{\ \overline{\beta }\overline{\delta }}^{\overline{\alpha }%
}+\mathbf{\Gamma }_{\ \overline{\beta }\overline{\gamma }}^{\overline{%
\varphi }}\ \mathbf{\Gamma }_{\ \overline{\varphi }\overline{\delta }}^{%
\overline{\alpha }}-\mathbf{\Gamma }_{\ \overline{\beta }\overline{\delta }%
}^{\overline{\varphi }}\ \mathbf{\Gamma }_{\ \overline{\varphi }\gamma }^{%
\overline{\alpha }}+\mathbf{\Gamma }_{\ \overline{\beta }\overline{\varphi }%
}^{\overline{\alpha }}W_{\overline{\gamma }\overline{\delta }}^{\overline{%
\varphi }}$ with N--adapted coefficients {\small
\begin{equation*}
\mathcal{R}_{~\overline{\beta }}^{\overline{\alpha }}=\{\mathbf{R}_{\
\overline{\beta }\overline{\gamma }\overline{\delta }}^{\overline{\alpha }%
}\}=\{\mathbf{R}_{\ \varepsilon \beta \gamma }^{\alpha }\mathbf{,R}_{\
B\beta \gamma }^{A}\mathbf{,R}_{\ \varepsilon \beta A}^{\alpha }\mathbf{,R}%
_{\ B\beta A}^{C}\mathbf{,R}_{\ \beta BA}^{\alpha },\mathbf{R}_{\
BEA}^{C}\}, \mbox{ for }
\end{equation*}%
\begin{eqnarray}
\mathbf{R}_{\ ebf}^{a} &=&\delta _{f}\mathbf{L}_{\ eb}^{a}-\delta _{b}%
\mathbf{L}_{\ eb}^{a}+\mathbf{L}_{\ eb}^{d}\mathbf{L}_{\ df}^{a}-\mathbf{L}%
_{\ ef}^{d}\mathbf{L}_{\ db}^{a}+\mathbf{L}_{\ ed}^{a}C_{bf}^{d}-\mathbf{B}%
_{\ eA}^{a}\Omega _{\ fb}^{A},  \notag \\
\mathbf{R}_{\ Bbf}^{A} &=&\delta _{f}\mathbf{L}_{\ Bb}^{A}-\delta _{b}%
\mathbf{L}_{Bf}^{A}+\mathbf{L}_{Bb}^{C}\mathbf{L}_{\ Cf}^{A}-\mathbf{L}%
_{Bf}^{C}\mathbf{L}_{\ Cb}^{A}+\mathbf{L}_{Bd}^{A}C_{bf}^{d}-\mathbf{B}_{\
BC}^{A}\Omega _{fb}^{C},  \notag \\
\mathbf{R}_{\ ebA}^{a} &=&\mathcal{V}_{A}\mathbf{L}_{\ eb}^{a}-\mathcal{D}%
_{b}\mathbf{B}_{eA}^{a}+\mathbf{B}_{\ eB}^{a}\mathbf{T}_{\ bA}^{B},  \notag
\\
\mathbf{R}_{\ BfA}^{C} &=&\mathcal{V}_{A}\mathbf{L}_{\ Bf}^{C}-\mathcal{D}%
_{\gamma }\mathbf{B}_{\ BA}^{C}+\mathbf{B}_{\ BD}^{C}\mathbf{T}_{\ \gamma
A}^{D},  \label{dcurv} \\
\mathbf{R}_{\ bBA}^{a} &=&\mathcal{V}_{A}\mathbf{B}_{\ bB}^{a}-\mathcal{V}%
_{B}\mathbf{B}_{bC}^{a}+\mathbf{B}_{\ bB}^{d}\mathbf{B}_{\ dC}^{a}-\mathbf{B}%
_{\ bC}^{d}\mathbf{B}_{\ dB}^{a},  \notag \\
\mathbf{R}_{\ ECB}^{A} &=&\mathcal{V}_{E}\mathbf{B}_{\ BC}^{A}-\mathcal{V}%
_{C}\mathbf{B}_{\ BE}^{A}+\mathbf{B}_{\ BC}^{F}\mathbf{B}_{\ FE}^{A}-\mathbf{%
B}_{\ BE}^{F}\mathbf{B}_{\ FC}^{A}.  \notag
\end{eqnarray}%
}
\end{itemize}
\end{theorem}

The formulas (\ref{dtors}) and (\ref{dcurv}) can be used for $\mathcal{D=}%
\widehat{\mathcal{D}},$ or $\mathcal{D=\tilde{D}}$ (in the last case, see
respective formulas (\ref{cdtors}) and (\ref{cdcurv})).

\subsection{N--adapted coefficients for the canonical d--connection}

\label{assectcan}

A d--metric structure $\overline{\mathbf{g}}=\{$ $\mathbf{g}_{\overline{%
\alpha }\overline{\beta }}\}$ on $\mathcal{T}^{E}\mathbf{P}$ is defined by
\begin{equation}
\overline{\mathbf{g}}=\mathbf{g}_{\overline{\alpha }\overline{\beta }}%
\mathbf{e}^{\overline{\beta }}\otimes \mathbf{e}^{\overline{\beta }}=\
\mathbf{g}_{ab}\ \mathcal{X}^{a}\otimes \mathcal{X}^{b}+\ \mathbf{g}_{AB}\
\delta ^{A}\otimes \delta ^{B}.  \label{dm}
\end{equation}%
It can be represented in generic off--diagonal form,\footnote{%
i.e. such metrics can not be diagonalized by coordinate transforms} $\mathbf{%
g}=g_{\overline{\alpha }\overline{\beta }}dz^{\overline{\alpha }}\otimes dz^{%
\overline{\beta }}$ where
\begin{equation}
\ g_{\overline{\alpha }\overline{\beta }}=\left[
\begin{array}{cc}
\ \mathbf{g}_{ab}+\mathcal{N}_{a}^{A}~\mathcal{N}_{b}^{B}\ \ \mathbf{g}_{AB}
& ~\mathcal{N}_{b}^{A}\ \ \mathbf{g}_{AC} \\
~\mathcal{N}_{a}^{E}\ \mathbf{g}_{ED} & \mathbf{g}_{DC}%
\end{array}%
\right] .  \label{offd}
\end{equation}%
Considering dual frame and/or frame transforms, $\mathbf{e}^{\overline{\beta
}}\rightarrow dz^{\overline{\beta }^{\prime }}$ and/or $\ \partial _{%
\overline{\alpha }^{\prime }}\rightarrow \mathbf{e}_{\overline{\alpha }}^{\
\overline{\alpha }^{\prime }}\partial _{\overline{\alpha }^{\prime }},$ with
$\ \mathbf{e}_{\overline{\alpha }}^{\ \overline{\alpha }^{\prime }}=\left[
\begin{array}{cc}
\mathbf{\ e}_{a}^{\ a^{\prime }} & \mathcal{N}_{a}^{B}\ \mathbf{e}_{B}^{\
A^{\prime }} \\
0 & \mathbf{e}_{A}^{\ A^{\prime }}%
\end{array}%
\right] ,$ we obtain quadratic relations between coefficients $\ g_{%
\overline{\alpha }\overline{\beta }}=e_{\overline{\alpha }}^{\ \overline{%
\alpha }^{\prime }}\ _{\overline{\beta }}^{\ \overline{\beta }^{\prime
}}\eta _{\overline{\alpha }^{\prime }\overline{\beta }^{\prime }},$ for $%
\eta _{\overline{\alpha }^{\prime }\overline{\beta }^{\prime }}=diag[\pm
1,...\pm 1]$ fixing a local signature for the metric field on $\mathcal{T}%
^{E}\mathbf{P}.$

Now, we can provide a proof of Theorem \ref{thcdc}. Let us consider $%
\widehat{\mathbf{\Gamma }}_{\ \overline{\beta }\overline{\gamma }}^{%
\overline{\alpha }}=(\widehat{\mathbf{L}}_{\beta \gamma }^{\alpha },\widehat{%
\mathbf{L}}_{B\gamma \;}^{A};\widehat{\mathbf{B}}_{\beta C}^{\alpha },%
\widehat{\mathbf{B}}_{BC}^{A}),$ where {\small
\begin{eqnarray}
\widehat{L}_{bf}^{a} &=&\frac{1}{2}\mathbf{g}^{ae}\left( \delta _{f}\mathbf{g%
}_{be}+\delta _{b}\mathbf{g}_{fe}-\delta _{e}\mathbf{g}_{bf}\right) +\frac{1%
}{2}\mathbf{g}^{ae}\left( \mathbf{g}_{bd}C_{\ fe}^{d}+\mathbf{g}_{fd}C_{\
eb}^{d}-\mathbf{g}_{ed}C_{\ bf}^{d}\right) ,  \notag \\
\widehat{L}_{Bf}^{A} &=&\mathcal{V}_{B}(\mathcal{N}_{f}^{A})+\frac{1}{2}%
\mathbf{g}^{AC}\left( \delta _{f}\mathbf{g}_{BC}-\mathbf{g}_{DC}\mathcal{V}%
_{B}(\mathcal{N}_{f}^{D})\ -\mathbf{g}_{DB}\mathcal{V}_{C}(\mathcal{N}%
_{f}^{D})\right) ,  \label{candcon} \\
\widehat{B}_{\beta C}^{\alpha } &=&\frac{1}{2}\mathbf{g}^{\alpha \tau }%
\mathcal{V}_{C}\mathbf{g}_{\beta \tau },\ \widehat{B}_{BC}^{A}=\frac{1}{2}%
\mathbf{g}^{AD}\left( \mathcal{V}_{C}\mathbf{g}_{BD}+\mathcal{V}_{B}\mathbf{g%
}_{CD}-\mathcal{V}_{D}\mathbf{g}_{BC}\right) .\   \notag
\end{eqnarray}%
}Using such values as N--adapted coefficients for a canonical d--connection $%
\widehat{\mathcal{D}}$, we can check that $\widehat{\mathcal{D}}\overline{%
\mathbf{g}}=0$ and that the $h$- and $v$-torsions (\ref{dtors}) are computed
$\widehat{T}_{\ bf}^{a}=C_{\ bf}^{a}$ and $\widehat{T}_{\ BC}^{A}=0.$ There
are nontrivial N--adapted coefficients of torsion of $\widehat{\mathcal{D}},$
i.e. $\widehat{T}_{\ bf}^{a},$ $\widehat{T}_{\ bA}^{a},\widehat{T}_{\
ba}^{A} $ and $\widehat{T}_{\ Ba}^{A},$ which can be computed by introducing
the coefficients (\ref{candcon}) into formulas (\ref{dtors}).

Let us show how we can compute the N--adapted coefficients of the distorting
relation $\widehat{\mathcal{D}}=\overline{\mathbf{\nabla }}+\widehat{%
\mathcal{Z}}$ in Remark \ref{remdistalg}. Having a d--metric structure $%
\overline{\mathbf{g}}$ on $\mathcal{T}^{E}\mathbf{P}$ we can always
construct a metric compatible Levi--Civita connection $\overline{\nabla }$
which is completely defined by the zero torsion condition, $~^{\nabla }%
\mathcal{T}^{\overline{\alpha }}=\{K_{\ \overline{\beta }\overline{\gamma }%
}^{\overline{\alpha }}\}=0.$ Parameterizing the N--adapted coefficients in
the form $K_{\ \overline{\beta }\overline{\gamma }}^{\overline{\alpha }}=(%
\overline{L}_{\beta \gamma }^{\alpha },\overline{L}_{B\gamma \;}^{A};%
\overline{B}_{\beta C}^{\alpha },\overline{B}_{BC}^{A}),$ we can verify via
straightforward computations with respect to (\ref{dderalg}) and (\ref%
{ddifalg}) that the conditions of the mentioned Remark us satisfied by
distortion relation
\begin{equation}
K_{\ \overline{\beta }\overline{\gamma }}^{\overline{\alpha }}=\widehat{%
\mathbf{\Gamma }}_{\ \overline{\beta }\overline{\gamma }}^{\overline{\alpha }%
}+\widehat{\mathbf{Z}}_{\ \overline{\beta }\overline{\gamma }}^{\overline{%
\alpha }},  \label{distrel1}
\end{equation}%
where the N--adapted coefficients of the distortion d-tensor $\widehat{%
\mathcal{Z}}=\{\widehat{\mathbf{Z}}_{\ \alpha \beta }^{\gamma }\}$ are%
{\small \ }%
\begin{eqnarray*}
\ \widehat{\mathbf{Z}}_{bf}^{a}&=&0,\ \widehat{\mathbf{Z}}_{bf}^{A} =-%
\widehat{\mathbf{B}}_{bB}^{a}\mathbf{g}_{af}\mathbf{g}^{AB}-\frac{1}{2}%
\Omega _{bf}^{A},~\widehat{\mathbf{Z}}_{Bf}^{a}=\frac{1}{2}\Omega _{af}^{C}%
\mathbf{g}_{CB}\mathbf{g}^{ba}-\Xi _{bf}^{ad}~\widehat{\mathbf{B}}_{dB}^{b},
\\
\widehat{\mathbf{Z}}_{Bf}^{A}&=&\ ^{+}\Xi _{CD}^{AB}~\widehat{\mathbf{T}}%
_{\gamma B}^{C},\widehat{\mathbf{Z}}_{BC}^{A}=0, \ \widehat{\mathbf{Z}}%
_{fB}^{a} =\frac{1}{2}\Omega _{bf}^{A}\mathbf{g}_{CB}\mathbf{g}^{ba}+\Xi
_{bf}^{ad}~\widehat{\mathbf{B}}_{dB}^{b},\   \notag \\
\widehat{\mathbf{Z}}_{bB}^{A}&=&-\ ^{-}\Xi _{CB}^{AD}~\widehat{\mathbf{T}}%
_{bD}^{C},\ \ \widehat{\mathbf{Z}}_{AB}^{a}=-\frac{\mathbf{g}^{ab}}{2}\left[
\widehat{\mathbf{T}}_{bA}^{C}\mathbf{g}_{CB}+\widehat{\mathbf{T}}_{bB}^{C}%
\mathbf{g}_{CA}\right] ,  \notag
\end{eqnarray*}%
for $\ \Xi _{bf}^{ad}~=\frac{1}{2}(\delta _{b}^{a}\delta _{f}^{d}-g_{\beta
\gamma }g^{\alpha \tau })$ and $~^{\pm }\Xi _{CD}^{AB}=\frac{1}{2}(\delta
_{C}^{A}\delta _{D}^{B}\pm g_{CD}g^{AB}).$

Introducing $K_{\ \overline{\beta }\overline{\gamma }}^{\overline{\alpha }}=%
\widehat{\mathbf{\Gamma }}_{\ \overline{\beta }\overline{\gamma }}^{%
\overline{\alpha }}$ (\ref{candcon}) into formulas (\ref{dcurv}), (\ref%
{driccialg}) and (\ref{sdcurv}), we compute respectively the coefficients of
curvature, $\widehat{\mathbf{R}}_{\ \overline{\beta }\overline{\gamma }%
\overline{\delta }}^{\overline{\alpha }}$,\ Ricci tensor, $\widehat{\mathbf{R%
}}_{\overline{\alpha }\overline{\beta }}$, and scalar curvature, $~^{s}%
\widehat{\mathbf{R}}.$ The distortions $K=\widehat{\mathbf{\Gamma }}+%
\widehat{\mathbf{Z}}$ (\ref{distrel1}) allows us to compute the distorting
tensors ($\widehat{\mathbf{Z}}_{\ \overline{\beta }\overline{\gamma }%
\overline{\delta }}^{\overline{\alpha }},\widehat{\mathbf{Z}}_{\overline{%
\alpha }\overline{\beta }}$ and $\ ^{s}\widehat{\mathbf{Z}})$ resulting in
similar values for the (pseudo) Riemannian geometry on $\mathcal{T}^{E}%
\mathbf{P}$ determined by $\left( \overline{\mathbf{g}},K\right) ,$ i.e. to
define $R_{\ \overline{\beta }\overline{\gamma }\overline{\delta }}^{%
\overline{\alpha }},R_{\ \overline{\beta }\overline{\gamma }}$ and $~^{s}R.$

\subsection{Lie algebroid mechanics and Kern--Matsumoto models}

Let us briefly outline some basic constructions \cite{vhamlagralg}
when the canonical N-- and d--connections and d--metric on $\mathcal{T}^{E}%
\mathbf{E,}$ for $\mathbf{P=E},$ can be generated from a regular Lagrangian $%
\mathcal{L}$ as a solution of the corresponding Euler--Lagrange equations.
The approach was developed in geometric mechanics with regular Lagrangians
on prolongations of Lie algebroids on bundle maps, see \cite%
{cortes,martinez1,dleon1} and references therein (the first models on
mechanics on algebroids were elaborated in \cite{weinstein,libermann}).

For a generating function $\mathcal{L}(x^{i},y^{a})\in C^{\infty }(\mathbf{E}%
)$ (or Lagrangian $L(x^{i},y^{a})$ if $\mathbf{E=TM)}$, we can compute $d^{E}%
\mathcal{L}=\rho _{a}^{i}(\partial _{i}\mathcal{L})\mathcal{X}^{a}+(\partial
_{A}\mathcal{L})\mathcal{V}^{A}$. A vertical endomorphism $S:\mathcal{T}^{%
\mathbf{E}}\mathbf{E}\rightarrow \mathcal{T}^{\mathbf{E}}\mathbf{E}$ is
constructed by $S(a,b,v)=\xi ^{V}(a,b)=(a,0,b_{a}^{V}).$ We consider $%
b_{a}^{V}$ as the vector tangent to the curve $a+\tau b;$ the curve
parameter $\tau =0.$ The vertical lift is a map $\xi ^{V}:\tau ^{\ast }%
\mathbf{E}\rightarrow \mathcal{T}^{\mathbf{E}}\mathbf{E}$ and the Liouville
dilaton vector field $\bigtriangleup (a)=\xi ^{V}(a,a)=(a,0,b_{a}^{V}).$
This allows us to construct a model of Lie algebroid mechanics for $\mathcal{%
L}$ which can be geometrized on $\mathcal{T}^{\mathbf{E}}\mathbf{E}$ in
terms of three geometric objects,%
\begin{eqnarray}
\mbox{ the Cartan 1-section:  }\theta _{\mathcal{L}}:= &&S^{\ast }(d\mathcal{%
L})\in Sec((\mathcal{T}^{\mathbf{E}}\mathbf{E})^{\ast });  \notag \\
\mbox{ the Cartan 2-section:  }\omega _{\mathcal{L}}:= &-d\theta _{\mathcal{L%
}}\in &Sec(\wedge ^{2}(\mathcal{T}^{\mathbf{E}}\mathbf{E})^{\ast });
\label{cartvar} \\
\mbox{ the Lagrangian energy : }E_{\mathcal{L}}:= &&\mathcal{L}%
i_{\bigtriangleup }\mathcal{L}-\mathcal{L}\in C^{\infty }(\mathbf{E}),
\notag
\end{eqnarray}%
where the Lie derivative $\mathcal{L}i_{\bigtriangleup }$ is considered in
the last formula. The dynamical equations for $\mathcal{L}$ are geometrized
\begin{equation}
i_{SX}\omega _{\mathcal{L}}=-S^{\ast }(i_{X}\omega _{\mathcal{L}})%
\mbox{ \
and \ }i_{\bigtriangleup }\omega _{\mathcal{L}}=-S^{\ast }(d\mathbf{E}_{%
\mathcal{L}}),\forall X\in Sec(\mathcal{T}^{\mathbf{E}}\mathbf{E}).
\label{geomeq}
\end{equation}

The geometric objects (\ref{cartvar}) and equations (\ref{geomeq}) can are
known for various applications in coefficient forms. Using local coordinates
$(x^{i},y^{a})\in \mathbf{E}$ and choosing a basis $\{\mathcal{X}_{a},%
\mathcal{V}_{A}\}\in Sec(\mathcal{T}^{\mathbf{E}}\mathbf{E}),$ for all $a,$
we have
\begin{eqnarray}
&&S\mathcal{X}_{a}=\mathcal{V}_{a},~S\mathcal{V}_{a}=0,~\bigtriangleup =y^{a}%
\mathcal{V}_{a},~E_{\mathcal{L}}=y^{a}\partial \mathcal{L}/\partial y^{a}-%
\mathcal{L},  \label{form1} \\
&&\omega _{\mathcal{L}}=\frac{\partial ^{2}\mathcal{L}}{\partial
y^{a}\partial y^{b}}\mathcal{X}^{a}\wedge \mathcal{V}^{b}+\frac{1}{2}(\rho
_{b}^{i}\frac{\partial ^{2}\mathcal{L}}{\partial x^{i}\partial y^{a}}-\rho
_{a}^{i}\frac{\partial ^{2}\mathcal{L}}{\partial x^{i}\partial y^{b}}%
+C_{ab}^{f}\frac{\partial \mathcal{L}}{\partial y^{f}})\mathcal{X}^{a}\wedge
\mathcal{X}^{b},  \notag
\end{eqnarray}%
for Lie algebroid structure functions $(\rho _{a}^{i},C_{ab}^{f}).$ As a
vertical endomorphism (equivalently, tangent structure) can be used the
operator $S:=\mathcal{X}^{a}\otimes \mathcal{V}_{a}.$ A regular system is
characterized by a non--degenerate Hessian
\begin{equation}
\tilde{g}_{ab}:=\frac{\partial ^{2}\mathcal{L}}{\partial y^{a}\partial y^{b}}%
,\ |\tilde{g}_{ab}|=\det |\tilde{g}_{ab}|\neq 0.  \label{hessian}
\end{equation}

An Euler--Lagrange section associated with $\mathcal{L}$ is given by any $%
\Gamma _{\mathcal{L}}=y^{a}\mathcal{X}_{a}+\varphi ^{a}\mathcal{V}_{a}\in
Sec(\mathcal{T}^{\mathbf{E}}\mathbf{E})$, when functions $\varphi
^{a}(x^{i},y^{b})$ solve this system of linear equations $\varphi ^{b}\frac{%
\partial ^{2}\mathcal{L}}{\partial y^{b}\partial y^{a}}+y^{b}(\rho _{b}^{i}%
\frac{\partial ^{2}\mathcal{L}}{\partial x^{i}\partial y^{a}}+C_{ab}^{f}%
\frac{\partial \mathcal{L}}{\partial y^{f}})-\rho _{a}^{i}\frac{\partial
\mathcal{L}}{\partial x^{i}}=0$. The semi--spray vector
\begin{equation}
\varphi ^{e}=\tilde{g}^{eb}(\rho _{b}^{i}\frac{\partial \mathcal{L}}{%
\partial x^{i}}-\rho _{a}^{i}\frac{\partial ^{2}\mathcal{L}}{\partial
x^{i}\partial y^{b}}y^{a}-C_{ba}^{f}\frac{\partial \mathcal{L}}{\partial
y^{f}}y^{a})  \label{semispray}
\end{equation}%
can be found in explicit forms for regular configurations when $\tilde{g}%
^{ab}$ is inverse to $\tilde{g}_{ab}.$ The section $\Gamma _{\mathcal{L}}$
transforms into a spray which states that the functions $\varphi ^{b}$ are
homogenous of degree $2$ on $y^{b}$ if the condition $\left[ \bigtriangleup
,\Gamma _{\mathcal{L}}\right] _{\mathbf{E}}=\Gamma _{\mathcal{L}}$ is
satisfied. The solutions of the Euler--Lagrange equations for $\mathcal{L},$
\begin{equation}
\frac{dx^{i}}{d\tau }=\rho _{a}^{i}y^{a}\mbox{ \ and \ }\frac{d}{d\tau }(%
\frac{\partial \mathcal{L}}{\partial y^{a}})+y^{b}C_{ab}^{f}\frac{\partial
\mathcal{L}}{\partial y^{f}}-\rho _{a}^{i}\frac{\partial \mathcal{L}}{%
\partial x^{i}}=0,  \label{eleq}
\end{equation}%
are parameterized by curves $c(\tau )=(x^{i}(\tau ),y^{a}(\tau ))\in \mathbf{%
E}.$

\subsection{The torsion and curvature of the normal d--connection}

Let us consider a 1--form associated to the normal d--connection $\widetilde{%
\mathcal{D}}=(\widetilde{\mathcal{D}}_{a},\widetilde{\mathcal{D}}_{A}),$ see
$\ ^{n}\mathcal{D}$ (\ref{ndc}) and $\ \widetilde{\mathbf{\Gamma }}_{%
\overline{\beta }\overline{\gamma }}^{\overline{\alpha }}=(\widehat{\mathbf{L%
}}_{bf}^{a},\widehat{\mathbf{B}}_{bc}^{a})$ (\ref{cdcc}),\ $\widetilde{%
\mathbf{\Gamma }}_{b}^{a}:=\widehat{\mathbf{L}}_{bf}^{a}\mathcal{X}^{f}+%
\widehat{\mathbf{B}}_{bc}^{a}\delta ^{c}$, where $\widetilde{\mathbf{e}}_{%
\overline{\alpha }}=(\widetilde{\mathbf{e}}_{a}=\delta _{a},\mathcal{V}_{A})$
and $\widetilde{\mathbf{e}}^{\overline{\beta }}:=\{\mathcal{X}^{a},\delta
^{b}=\mathcal{V}^{b}+\widetilde{\mathcal{N}}_{f}^{b}\mathcal{V}^{f}\}$ are
taken as in (\ref{ddcanad}). We can prove that the Cartan structure
equations are satisfied,%
\begin{eqnarray}
d\mathcal{X}^{a}-\mathcal{X}^{b}\wedge \widetilde{\mathbf{\Gamma }}%
_{b}^{a}&=&-h\widetilde{\mathcal{T}}^{a},\ d\delta ^{c}-\delta ^{b}\wedge
\widetilde{\mathbf{\Gamma }}_{b}^{c}=-v\widetilde{\mathcal{T}}^{c},
\label{cart1} \\
\mbox{ and } d\widetilde{\mathbf{\Gamma }}_{b}^{a}-\widetilde{\mathbf{\Gamma
}}_{b}^{c}\wedge \widetilde{\mathbf{\Gamma }}_{c}^{a}&=& -\widetilde{%
\mathcal{R}}_{\ b}^{a}.  \label{cart2}
\end{eqnarray}

The h-- and v--components of the torsion 2--form $\widetilde{\mathcal{T}}%
^{\alpha }=\left( h\widetilde{\mathcal{T}}^{a},v\widetilde{\mathcal{T}}%
^{a}\right) =\widetilde{\mathbf{T}}_{\ cb}^{a}\delta ^{c}\wedge \delta ^{b}$
from (\ref{cart1}). The N--adapted coefficients are computed
\begin{equation}
h\widetilde{\mathcal{T}}^{a}=\widetilde{\mathbf{B}}_{bc}^{a}\mathcal{X}%
^{b}\wedge \delta ^{c},v\widetilde{\mathcal{T}}^{a}=\frac{1}{2}\tilde{\Omega}%
_{bc}^{a}\mathcal{X}^{b}\wedge \mathcal{X}^{c}+(\mathcal{V}_{c}\widetilde{%
\mathcal{N}}_{b}^{a}-\widetilde{\mathbf{L}}_{bc}^{a})\mathcal{X}^{b}\wedge
\delta ^{c},  \label{tform}
\end{equation}%
where $\tilde{\Omega}_{bc}^{a}$ are coefficients of the canonical
N--connection curvature (\ref{cncalgebr}) of the canonical N--connection, $%
\mathcal{N}_{f}^{b}\rightarrow $ $\widetilde{\mathcal{N}}_{f}^{b}$ (\ref%
{canonnc}). In explicit form, the N--adapted coefficients of d--torsion (\ref%
{dtors}) of $\widetilde{\mathcal{D}}$ are
\begin{equation}
\widetilde{\mathbf{T}}_{\ cb}^{a}=0,\widetilde{\mathbf{T}}_{\ cB}^{a}=%
\widetilde{\mathbf{B}}_{cB}^{a},\widetilde{\mathbf{T}}_{\ cb}^{A}=\tilde{%
\Omega}_{bc}^{A},\widetilde{\mathbf{T}}_{\ cB}^{A}=\mathcal{V}_{B}\widetilde{%
\mathcal{N}}_{c}^{A}-\widetilde{\mathbf{L}}_{cB}^{A},\widetilde{\mathbf{T}}%
_{\ CB}^{A}=0,  \label{cdtors}
\end{equation}%
where indices $A,B,C,...$ are transformed into, respectively, $a,b,c,...$ at
the end (in order to keep the convention on $h$- and $v$--indices). We note
that we have chosen such nonholonomic distributions when the $Y$--fields
from the Theorem \ref{setscdc} are such way stated that the formulas (\ref%
{cdtors}) on $\mathcal{T}^{\mathbf{E}}\mathbf{E}$ are symilar to those on $%
\mathbf{TM}.$

Using a differential form calculus, we can also compute the curvature
2--form (\ref{cart2}),{\small
\begin{equation}
\widetilde{\mathcal{R}}_{\ f}^{e}=\widetilde{\boldsymbol{R}}_{\
fab}^{e}\delta ^{a}\wedge \delta ^{b}=\frac{1}{2}\widetilde{\mathcal{R}}_{\
fab}^{e}\mathcal{X}^{a}\wedge \mathcal{X}^{b}+\widetilde{\mathcal{P}}_{\
faB}^{e}\mathcal{X}^{a}\wedge \delta ^{B}+\frac{1}{2}\widetilde{\mathcal{S}}%
_{\ fAB}^{e}\delta ^{A}\wedge \delta ^{B},  \label{cform}
\end{equation}%
} where the nontrivial N--adapted coefficients are {\small
\begin{eqnarray}
\widetilde{\mathcal{R}}_{\ fab}^{e} &=&\delta _{b}\widetilde{\mathbf{L}}%
_{fa}^{e}-\delta _{a}\widetilde{\mathbf{L}}_{fb}^{e}+\widetilde{\mathbf{L}}%
_{fa}^{d}\widetilde{\mathbf{L}}_{db}^{e}-\widetilde{\mathbf{L}}_{fb}^{d}%
\widetilde{\mathbf{L}}_{da}^{e}-\widetilde{\mathbf{B}}_{fA}^{e}\tilde{\Omega}%
_{ba}^{A},  \label{cdcurv} \\
\widetilde{\mathcal{P}}_{\ bfB}^{a} &=&\mathcal{V}_{B}\widehat{L}_{bf}^{a}-%
\widetilde{\mathcal{D}}_{f}\widetilde{\mathbf{B}}_{bB}^{a},\ \widetilde{%
\mathcal{S}}_{\ fAB}^{e} =\mathcal{V}_{B}\widetilde{\mathbf{B}}_{fA}^{e}-%
\mathcal{V}_{A}\widetilde{\mathbf{B}}_{fB}^{e}+\widetilde{\mathbf{B}}%
_{fA}^{d}\widetilde{\mathbf{B}}_{dB}^{e}-\widetilde{\mathbf{B}}_{fB}^{d}%
\widetilde{\mathbf{B}}_{dA}^{e}.  \notag
\end{eqnarray}
}

The distortion relations for the Ricci tensor are computed
\begin{eqnarray}
R_{\overline{\alpha }\overline{\beta }}&=&\widetilde{\mathbf{R}}_{\overline{%
\alpha }\overline{\beta }}+\widetilde{\mathbf{Z}}_{\overline{\alpha }%
\overline{\beta }},  \label{driccidist} \\
R_{\overline{\beta }\overline{\gamma }} &=&R_{\ ~\overline{\beta }\overline{%
\gamma }\overline{\alpha }}^{\overline{\alpha }}=\mathbf{e}_{\overline{%
\delta }}K_{\ \overline{\beta }\overline{\gamma }}^{\overline{\alpha }}-%
\mathbf{e}_{\overline{\gamma }}K_{\ \overline{\beta }\overline{\delta }}^{%
\overline{\alpha }}+K_{\ \overline{\beta }\overline{\gamma }}^{\overline{%
\varphi }}K_{\ \overline{\varphi }\overline{\delta }}^{\overline{\alpha }%
}-K_{\ \overline{\beta }\overline{\delta }}^{\overline{\varphi }}K_{\
\overline{\varphi }\gamma }^{\overline{\alpha }}+K_{\ \overline{\beta }%
\overline{\varphi }}^{\overline{\alpha }}W_{\overline{\gamma }\overline{%
\delta }}^{\overline{\varphi }},  \notag \\
\widetilde{\mathbf{Z}}_{\ \overline{\beta }\overline{\gamma }} &=&\widetilde{%
\mathbf{Z}}_{\ \overline{\beta }\overline{\gamma }\overline{\alpha }}^{%
\overline{\alpha }}=\mathbf{e}_{\overline{\alpha }}\widetilde{\mathbf{Z}}_{\
\overline{\beta }\overline{\gamma }}^{\overline{\alpha }}-\mathbf{e}_{%
\overline{\gamma }}\ \widetilde{\mathbf{Z}}_{\ \overline{\beta }\overline{%
\alpha }}^{\overline{\alpha }}+\widetilde{\mathbf{Z}}_{\ \overline{\beta }%
\overline{\gamma }}^{\overline{\varphi }}\ \widetilde{\mathbf{Z}}_{\
\overline{\varphi }\overline{\alpha }}^{\overline{\alpha }}-\widetilde{%
\mathbf{Z}}_{\ \overline{\beta }\overline{\alpha }}^{\overline{\varphi }}\
\widetilde{\mathbf{Z}}_{\ \overline{\varphi }\overline{\gamma }}^{\overline{%
\alpha }}+  \notag \\
&&\widetilde{\mathbf{\Gamma }}_{\ \overline{\beta }\overline{\gamma }}^{%
\overline{\varphi }}\ \widetilde{\mathbf{Z}}_{\ \overline{\varphi }\overline{%
\alpha }}^{\overline{\alpha }}-\widetilde{\mathbf{\Gamma }}_{\ \overline{%
\beta }\overline{\alpha }}^{\overline{\varphi }}\ \widetilde{\mathbf{Z}}_{\
\overline{\varphi }\overline{\gamma }}^{\overline{\alpha }}+\widetilde{%
\mathbf{Z}}_{\ \overline{\beta }\overline{\gamma }}^{\overline{\varphi }}%
\widetilde{\mathbf{\Gamma }}_{\ \overline{\varphi }\overline{\alpha }}^{%
\overline{\alpha }}-\widetilde{\mathbf{Z}}_{\ \overline{\beta }\overline{%
\alpha }}^{\overline{\varphi }}\ \widetilde{\mathbf{\Gamma }}_{\ \overline{%
\varphi }\overline{\gamma }}^{\overline{\alpha }}+\widetilde{\mathbf{Z}}_{\
\overline{\beta }\overline{\varphi }}^{\overline{\alpha }}W_{\overline{%
\gamma }\overline{\alpha }}^{\overline{\varphi }}.  \notag
\end{eqnarray}

\end{document}